\documentclass[format=acmsmall]{acmart}

\usepackage{amsmath,amsthm,amssymb,amsfonts}
\usepackage{graphicx}
\usepackage[linguistics]{forest}
\usepackage{soul}
\usepackage{natbib}
\hypersetup{final} 

\newcommand{\JSR}{\operatorname{JSR}}

\newcommand{\coast}{\operatorname{co}_{\ast}}
\newcommand{\comin}{\operatorname{co}_-}
\newcommand{\norm}[1]{\left\lVert#1\right\rVert}
\newcommand{\vardot}{\mathord{\,\cdot\,}} 

\newcommand{\numberthis}{\addtocounter{equation}{1}\tag{\theequation}}
\renewcommand\bf{\bfseries}

\newcommand{\fwsym}[1]{\parbox[c]{.9em}{\protect\ensuremath{#1}}}
\newcommand{\ppmm}{\raisebox{-0.05ex}{\fwsym{\pm}}}
\newcommand{\Circ}{{\raisebox{-0.2ex}{\scalebox{1.4}{$\circ$}}}}
\newcommand{\pp}{\fwsym{+}}
\newcommand{\mm}{\raisebox{-0.24ex}{\fwsym{-}}}
\newcommand{\nn}{\fwsym{\Circ}}

\newcommand{\simin}{\mathop{\in}\limits^{\vbox to -.4ex{\kern-0.6ex\hbox{\scriptsize$\sim$}\vss}}}

\newcommand{\RR}{{\mathbb R}}
\newcommand{\NN}{{\mathbb N}}
\newcommand{\ZZ}{{\mathbb Z}}

\theoremstyle{acmplain}
\newtheorem{theorem}{Theorem}[section]
\newtheorem{lemma}[theorem]{Lemma}

\theoremstyle{acmdefinition}

\newtheorem{algorithm}[theorem]{Algorithm}
\newtheorem{example}[theorem]{Example}
\newtheorem{remark}[theorem]{Remark}

\begin{document}

\acmJournal{TOMS}
\title{Algorithm xxx: Improved  invariant polytope algorithm and applications}
\author{Thomas Mejstrik}
\affiliation{%
    \department{Faculty of Mathematics}
    \institution{University of Vienna}
    \streetaddress{Universit\"atsring 1}
    \postcode{1010 Vienna}
    \country{Austria}        
}
\email{tommsch@gmx.at}
\position{PostDoc}



\setcopyright{acmcopyright}

\begin{abstract}
In several papers of 2013~--~2016, Guglielmi and Protasov made a 
breakthrough in the problem of the joint spectral radius computation, developing
the invariant polytope algorithm which for most matrix families 
finds the exact value of the joint spectral radius. This algorithm found 
many applications in problems of functional analysis, approximation theory, 
combinatorics, etc.. In this paper we propose a modification of the invariant polytope 
algorithm making it roughly 3 times faster (single threaded), 
suitable for higher dimensions and parallelise it.
The modified version works for most matrix families of dimensions up to 25, 
for non-negative matrices up to 3000.
Besides we introduce a new, fast algorithm, called modified Gripenberg algorithm, 
for computing good lower bounds for the joint spectral radius.
The corresponding examples and statistics of numerical results are provided.  
Several applications of our algorithms are presented. In particular, 
we find the exact values of the regularity exponents of Daubechies wavelets up to order 42
and the capacities of codes that avoid certain difference patterns.
\end{abstract}

\maketitle

\section{Introduction and notation}

The \emph{joint spectral radius} ($\JSR$) of a set of matrices is a quantity which describes the maximal asymptotic growth rate of the norms of products of matrices from that set (with repetitions permitted). Precisely,
given a finite set $\mathcal{A}=\{A_j:j=1,\ldots,J\}\subseteq\RR^{s\times s}$, $s\in\NN$, then
\begin{equation}\label{equ_jsr}
\JSR(\mathcal{A}):=
\lim_{n\rightarrow\infty}
\max_{A_j\in\mathcal{A}}\left\| A_{j_n}\cdots A_{j_2}A_{j_1}\right\|^{1/n}.
\end{equation}
In~\cite{BW92} it is proved that (for finite $\mathcal{A}$)
\begin{equation}\label{equ_generalized_jsr}
\JSR(\mathcal{A})=
\limsup_{n\rightarrow\infty,\, A_j\in\mathcal{A}}
\rho( A_{j_n}\cdots A_{j_2}A_{j_1})^{1/n},
\end{equation}
where $\rho$ is the classical \emph{spectral radius} of a matrix.
With $\#\mathcal{A}$ we denote the \emph{number of elements} of the set $\mathcal{A}$.
If $\#\mathcal{A}=1$, then the $\JSR$ reduces to the spectral radius of a matrix.

The $\JSR$ has been defined in~\cite{Rota60} and since appeared in many (seemingly unrelated) mathematical applications, e.g.\  for
computing the regularity of wavelets and of subdivision schemes~\cite{DL1992}, 
the capacity of codes~\cite{MOS01},
the stability of linear switched systems~\cite{Gur95} or
in connection with the Euler partition function~\cite{Prot00}.

The computation of the $\JSR$ is a notoriously hard problem. 
Even for non-negative matrices with rational coefficients 
this problem is NP-hard~\cite{BT97}. Moreover,  
the question whether $\JSR(\mathcal{A})\leq1$ for a given set $\mathcal{A}$ is algorithmically  undecidable~\cite{BT00}.
Most algorithms which try to compute or to approximate the $\JSR$ make use of the inequality~\cite{DL1992}
\begin{equation}\label{equ_bound_jsr}
\max_{A_{j}\in\mathcal{A}}\rho\left( A_{j_k}\cdots A_{j_1}\right)^{1/k}
\leq
\JSR(\mathcal{A})
\leq
\max_{A_j\in\mathcal{A}}
\big\| 
A_{j_k}\cdots A_{j_1}
\big\|^{1/k},
\end{equation}
which holds for any $k\in\NN$. For a product $A_{j_k}\cdots A_{j_1}$ we say the number $\rho\left( A_{j_k}\cdots A_{j_1}\right)^{1/k}$ is its \emph{normalized spectral radius}. 
Equation~\eqref{equ_bound_jsr} tells us that the normalized spectral radius of every product 
is a valid lower bound for the $\JSR$, on the contrary, 
one has to compute the norms of all products of a fixed length $k\in\NN$ to obtain a valid upper bound.

If there exists a product $\Pi=A_{j_n}\cdots A_{j_1}$, $A_j\in\mathcal{A}$, such that $\rho(\Pi)^{1/n}=\JSR(\mathcal{A})$, we call the product a \emph{spectral maximizing product (s.m.p.)}.
It has been shown that 
there exist sets of matrices such that the normalized spectral radius of every finite product is strictly less than the $\JSR$~\cite{HMST11}. In other words, not all sets of matrices posses an s.m.p..
It is an open question whether pairs of binary matrices always posses an s.m.p.~\cite{BJ08}.

An s.m.p.\ is called \emph{dominant} if there exists $\gamma>0$ 
such that $\gamma<\JSR(\mathcal{A})$ and 
$\rho({A}_{j_l}\cdots {A}_{j_1})^{1/l}<\gamma$ whenever 
${A}_{j_l}\cdots {A}_{j_1}$ is not an s.m.p..
Dominant s.m.p.s play a role for the termination of the invariant polytope algorithm 
discussed in Sections~\ref{sec_invpoly} and~\ref{sec_modinvpoly}.

There are three common strategies to exploit Inequality~\eqref{equ_bound_jsr}:
$(i)$ Compute all products up to a length $k\in\NN$~\cite{Grip96,MOS01};
$(ii)$ Take a suitable family of norms and minimize the right hand side of~\eqref{equ_bound_jsr} 
with respect to that family~\cite{AJPR11,BN05,PJ08,BJP10};
$(iii)$ Construct a norm which gives good estimates in~\eqref{equ_bound_jsr} for short products, 
preferably for products of length one~\cite{GP13,GP16,GWZ05,GZ08,GZ09,Koz10b}. 
The Gripenberg algorithm~\cite{Grip96}, discussed in Section~\ref{sec_modgrip}, belongs to class $(i)$,
the invariant polytope algorithm~\cite{GP13,GP16,GWZ05,GZ08,GZ09}, discussed in Section~\ref{sec_invpoly}, to class $(iii)$.
The invariant polytope algorithm is, up to now, one of only two algorithms which can can compute the exact value of the $\JSR$ for a large number of matrix families.
The second one is the infinite tree algorithm~\cite{MR01} which does not belong to any of the classes above.
In this paper we concentrate on the invariant polytope algorithm.

We will call 
a norm $\|\cdot\|$ \emph{extremal} for $\mathcal{A}$ if 
\begin{equation}
\|A_j x\|\leq \JSR(\mathcal{A})\cdot\|x\| \text{ for all } x\in\RR^s \text{ and for all } A_j\in\mathcal{A}.
\end{equation}

In~\cite{Bar88} it is shown that every \emph{irreducible} family of matrices, 
i.e.\ a family of matrices which have no trivial common invariant subspaces, possesses an extremal norm. 
Its construction is easily described in terms of the set 
\begin{equation}\label{equ_invariantbody}
P(v)=
\operatorname{co} 
\bigcup_{n\in\NN_0,\ {A}_j\in{\mathcal{A}}} 
\left\{\pm\,{A}_{j_n}\cdots {A}_{j_1} v \right\},
\end{equation}
where $\operatorname{co}$ denotes the \emph{convex hull} and $v\in\RR^s$.

\begin{theorem}\cite{GZ08,BW92}.
\label{thm_invariantbody}
If ${\mathcal{A}}$ is irreducible, $\JSR(\mathcal{A})\geq 1$ and for a given $v\in\RR^s$ the set $P(v)$ is bounded and has non-empty interior, 
then $\JSR({\mathcal{A}})=1$ and $P(v)$ is the unit ball of an extremal norm $\|\cdot\|_{P(v)}$ for ${\mathcal{A}}$.

Conversely, if ${\mathcal{A}}$ is irreducible and $\JSR({\mathcal{A}})=1$, then for any $v\in\RR^s$, $P(v)$ is a bounded subset of $\RR^s$.
\end{theorem}

Clearly, the unit ball of a norm completely describes the corresponding norm.
Given $P\subseteq\RR^s$, a closed, convex and \emph{balanced} 
($\alpha P\subseteq P$ for all $|\alpha|<1$) 
body with non-empty interior, the so-called \emph{Minkowski norm} 
$\|\cdot\|_P:\RR^s\rightarrow\RR$,
\begin{equation}
\|\cdot\|_P=\inf\{r>0:x\in r P \}
\end{equation}
fulfils $\{x\in\RR^s:\|x\|_P\leq1\}=P$.

The idea of the invariant polytope algorithm~\ref{alg_invpoly} is to construct an invariant set $P$ for the matrices in 
the set $\mathcal{A}$ in finitely many steps. This is possible when $P$ is a polytope.

We will describe polytopes by the convex hull of its vertices.
For finite $V\subseteq\RR^s$ we define the \emph{symmetrized convex hull} of $V$  by
\begin{equation}
\operatorname{co}_s V=
\Big%
\{
x\in \RR^s: x=\sum_{v\in V} t_v v 
\quad\text{with}\quad \sum_{v\in V}|t_v|\leq 1,\ t_v\in \RR^s
\Big%
\}
=
\operatorname{co}(V\cup -V).
\end{equation}
For finite $V\subseteq\RR^s_+$ we define the \emph{cone} of $V$ with respect to the first orthant by
\begin{equation}
\operatorname{co}_{-}V=
\big\{
x\in \RR^s_+ : x=y-z,\ y\in\operatorname{co}(V),\ z\in\RR^s_+
\big\}
.
\end{equation}
For simplicity, we denote with $\coast V$ any of these convex hulls 
($\operatorname{co}$, 
$\operatorname{co}_s$, 
$\operatorname{co}_{-}$)
depending on the context
.

In all cases we identify a (finite) set $V$ with the matrix whose columns are the coordinates of the points $v\in V$. 
Lemma~\ref{thm_estimate_1} shows 
properties of Minkowski norms corresponding to 
various 
convex hulls.

\begin{lemma}\label{thm_estimate_1}
    Let $V\subseteq\RR^s$
    and $x\in\RR^s$. Then
    \begin{enumerate}
        \item If $W$ are the vertices of another central symmetric polytope with non-empty interior such that 
        $\coast W\subseteq \coast V$, then
        $\|\cdot\|_{\coast V}\leq \|\cdot\|_{\coast W}$.
        
        \item $\|x\|_{\coast V}\leq\|t\|_1\leq\sqrt{m}\|t\|_2$, where $Vt=x$, $t\in\RR^m$.
        
        \item $\|x\|_{\operatorname{co}_s V} \geq \|V^+x\|_2$, 
        where $V^+$ is the Moore-Penrose pseudo-inverse of $V$.
        
        \item If there exists $w\in\RR^s$ such that 
        $|\langle w,v\rangle| < |\langle w,x\rangle|$ for all $v\in V$, 
        then $x\not\in\operatorname{co}_{s} V$.
        
        \item \label{thm_estimate_1_pos}
        If $V\subseteq\RR^s_+$, $x\in\RR^s_+$ 
        and there exists $v\in V$ such that $x_l\leq v_l$ for all $l=1,\ldots,s$, 
        then $x\in \operatorname{co}_{-} V$.
        
        \item\label{thm_estimate_2_pos}
        If $V\subseteq\RR^s_+$, $x\in\RR^s_+$ and there exists $l\in\{1,\ldots,s\}$ 
        such that $x_l > v_l$ for all $v\in V$, 
        then $x\notin\operatorname{co}_{-} V$.

    \end{enumerate}
\end{lemma}
\begin{proof}
    \begin{enumerate}
        \item This immediately follows from the definition of the Minkowski norm.
        
        \item 
        Let $x\in\operatorname{co}_s V$ and $t\in\RR^{\#V}$ such that $x=Vt$.
        Define $\tilde{x}=\frac{x}{\|x\|_{\operatorname{co}_s V}}\in\partial\operatorname{co}_s V$ and 
        $\tilde{t}=\frac{t}{\|x\|_{\operatorname{co}_s V}}$.
        It follows that $\tilde{x}=V\tilde{t}$ with $\|\tilde{t}\|_1\geq 1$.
        Indeed, $\|\tilde{t}\|_1<1$ would imply that $\tilde{x}\in\left(\operatorname{co}_s V\right)^\circ$. Clearly,
        $1=\|\tilde{x}\|_{\operatorname{co}_s V}\leq \|\tilde{t}\|_1$ and
        $\|x\|_{\operatorname{co}_s V}\leq \|t\|_1$.
        Finally, by $(1)$, $\|x\|_{\operatorname{co}_- V} \leq \|x\|_{\operatorname{co}_s V}$.
        The second inequality follows from the equivalence of norms.
        
        \item 
        Let $x\in\operatorname{co}_s V$ and $t\in\RR^{\#V}$ such that $x=Vt$.
        It follows that $\|t\|_1\geq \|t\|_2\geq \|V^+x\|_2$ because, 
        by construction of the Moore-Penrose pseudo inverse, 
        $V^+ x$ is the unique solution to $Vt=x$ with minimum 2-norm.
        Finally, by ~$(2)$, $\|x\|_{\operatorname{co}_s V}=\min_{t\in\RR^{\# V}: Vt=x} \|t\|_1\geq \|V^+x\|_2$.
        
        \item 
        If  $|\langle w,v\rangle| < |\langle w,x\rangle| $ for all $v\in V$, 
        then there exists a hyperplane which separates the point $x$ and the polytope 
        $\operatorname{co}_{s} V$. From this the claim directly follows.
        
        \item Defining $z:=v-x$ we see that $z\in \RR^s_+$ which implies $x=v-z\in\operatorname{co}_{-}  V$.
        
        \item Since $\operatorname{co}_s V$ is convex, $y_l\leq v_l$ for all $y\in \operatorname{co} V\cap\RR^s_+$. 
        Since $z_l>0$ it follows that $y_l-z_l\leq y_l \leq v_l < x_l$.
        Thus, there does not exist $y,z$ such that $x=y-z$.
    \end{enumerate}
\end{proof}

\begin{remark}
    Estimate~\ref{thm_estimate_1}~\eqref{thm_estimate_1_pos} uses the fact that the norms $\|\cdot\|_{\operatorname{co}_{-}V}$ are orthant monotonic,
    i.e.\ %
    $
    \|x\|_{\operatorname{co}_{-}V}
    \leq
    \|y\|_{\operatorname{co}_{-}V}
    $, $x,y\in\RR^s_+$
    whenever
    $0 \leq x_i \leq  y_i$
    for all $i=1,\ldots,s$.
    It would be interesting to know whether and when Minkowski norms
    $\|\cdot\|_{\operatorname{co}_s V}$ 
    or Minkowski norms composed with linear mappings 
    $\|M\cdot\|_{\operatorname{co}_s V}$, $M\in\RR^{s\times s}$,
    are orthant monotonic.
    This would allow to transfer the estimate~\ref{thm_estimate_1}~\eqref{thm_estimate_1_pos} to Minkowski norms corresponding to symmetrised convex hulls $\operatorname{co}_s V$.
\end{remark}

\subsection{Invariant polytope algorithm and outline for the paper}\label{sec_invpoly}
In this section we present the idea of the invariant polytope algorithm.
The major topic of this paper are modifications to the invariant polytope algorithm making it
\begin{itemize}
\item faster and parallel,
\item more robust and
\item more efficient for larger matrices.
\end{itemize}
These modifications are outlined in Section~\ref{sec_outline}. 
The actual modified invariant polytope algorithm~\ref{alg_modinvpoly} is given in Section~\ref{sec_modinvpoly}.
In Section~\ref{sec_modgrip} we introduce the modified Gripenberg algorithm~\ref{alg_modgrip} 
which is capable of finding very long s.m.p.-candidates in short time.
Section~\ref{sec_examples} is devoted to numerical examples showing where the algorithms from 
Sections~\ref{sec_modgrip} and~\ref{sec_modinvpoly} perform well and where they are not applicable any more. 

\begin{algorithm}[Invariant polytope algorithm~\cite{GP13,GP16,GWZ05,GZ08,GZ09}]\label{alg_invpoly}~
Given 
$\mathcal{A}=\{A_j:j=1,\ldots,J\}\subseteq\RR^{s\times s}$.
\begin{enumerate}
\item\label{alg_invpoly_cand}
For some $D\in\NN$ look over all products of matrices in $\mathcal{A}$ of length less than $D$ 
and choose a shortest product $\Pi_1$ such that $\rho_c:=\rho(\Pi_1)^{1/l_1}$ is maximal, 
where $l_1$ is the length of the product and call $\Pi_1$ 
\emph{spectral maximizing product-candidate} (\emph{s.m.p.-candidate}). 
Set $\tilde{\mathcal{A}}:=\rho_c^{-1}\mathcal{A}$.
Now we try to prove that $\JSR(\tilde{\mathcal{A}})\leq 1$. 

\item\label{alg_invpoly_leadingeigenvector}
Let $v_1$ be the \emph{leading eigenvector} of $\Pi_1$
i.e.\ the eigenvector with respect to the largest eigenvalue in magnitude

\item\label{alg_invpoly_cyclicroot}
Construct the cyclic root $\mathcal{H}$: Let $v^{(i)}_1$, $i=1,\ldots,l_1-1$, 
be the leading eigenvectors of the cyclic permutations of $\tilde{\Pi}_1$, 
i.e.\ for $\tilde{\Pi}_1=\tilde{A}_{j_{l_1}}\cdots \tilde{A}_{j_1}$ 
we get $v^{(i)}_1:=\tilde{A}_{j_i}\cdots \tilde{A}_{j_1}v_1$.
Set $\mathcal{H}:=\{v^{(0)}_1,\ldots,v^{(l_1-1)}_1\}$ and $V:=\mathcal{H}$.

\item \label{alg_invpoly_testinside}
For all $v\in V$ and for all $j=1,\ldots,J$
\begin{itemize}
\item[]
If  $\|\tilde{A}_j v\|_{\coast V}>1$ set $V:= V \cup \tilde{A}_j v$.

Depending on $\mathcal{A}$ and the leading eigenvector $v_0$ we use different convex hulls:

\emph{case $(P)$}: If all entries of the matrices $A_j$ are non-negative, 
then we can take a non-negative leading eigenvector~$v_0$ in 
step~\eqref{alg_invpoly_leadingeigenvector}
and use $\comin$.

\emph{case $(R)$}: If the matrices $A_j$ have positive and negative entries 
and the leading eigenvector~$v_0$ is real, then we use $\operatorname{co}_s$;


\end{itemize}
\item
Repeat step~\eqref{alg_invpoly_testinside} until $\tilde{\mathcal{A}}V\subseteq\coast V$.
\item If $\tilde{\mathcal{A}}V\subseteq\coast V$, then the algorithm terminates and we have found an 
invariant polytope~$\coast V$, 
which implies that $\|\tilde{A}_j\|_{\coast V}\leq 1$ for all $j=1,\ldots,J$, 
or in other words, $\JSR(\tilde{\mathcal{A}})\leq 1$.
\end{enumerate}

\end{algorithm}

\begin{remark}
In step~\ref{alg_invpoly}~\eqref{alg_invpoly_testinside} we actually add a vertex 
$\tilde{A}_j v\not\in\mathcal{H}$ even if it lies slightly inside of the polytope, 
i.e.\ if $\|\tilde{A}_j v\|_{\coast V}>1-\epsilon$, 
where $\epsilon>0$ is the accuracy up to which the norm can be computed. 
This is important to obtain a mathematically rigorous result.
\end{remark}

\begin{remark}
    If neither \emph{case $(P)$} nor \emph{case $(R)$} applies, i.e.\ 
    the matrices are not strictly non-negative and have complex leading eigenvalues, 
    then one would have to consider complex polytopes, which is not discussed in this paper.
\end{remark}

Figure~\ref{fig_invpoly} presents the invariant polytope algorithm on some concrete example.

\begin{figure}[ht]
{
    \centering
    \includegraphics{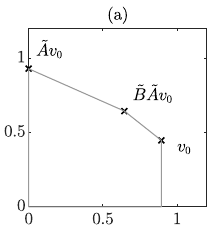}
    \includegraphics{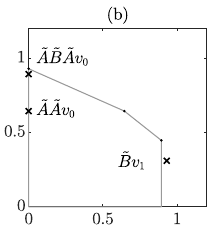}
    \includegraphics{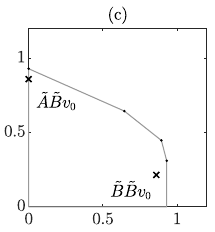}
    \caption[]%
    {The polytope $\operatorname{co}_- V$ 
        as constructed by the invariant polytope algorithm~\ref{alg_invpoly} 
        for the matrices 
        $A=\left[\begin{smallmatrix}0&0\\1&1\end{smallmatrix}\right]$,
        $B=\left[\begin{smallmatrix}1&1\\0&1\end{smallmatrix}\right]$,
        In $(a)$ we see the cone $\comin \mathcal{H}$ with respect to the cyclic root
        $\mathcal{H}=\big\{v_1,\tilde{A}v_1, \tilde{B}\tilde{A}v_1\big\}$.
        In $(b)$ we see the vertices $\tilde{A}v_1$, $\tilde{A}\tilde{A}v_1$ and $\tilde{A}\tilde{B}\tilde{A}v_1$
        constructed in the first iteration.
        In $(c)$ we see the new polytope $\comin (\mathcal{H}\cup\tilde{B}v_1)$
        together with the vertices  $\tilde{B}\tilde{B}v_1$ and $\tilde{A}\tilde{B}v_1$
        constructed in the second iteration, which are all mapped into the interior of
        $\operatorname{co}_- \mathcal{H}\cup\tilde{B}v_1$.
More precise:    
\newline
All entries of $A$ and $B$ are non-negative, thus, we are in case $(P)$ and use the cone hull
$\operatorname{co}_-$ to compute the Minkowski-norms in step~\ref{alg_invpoly}~\ref{alg_invpoly_testinside}.
\newline
    {($1$)} We choose 
    $\Pi_1=BBA$,
    which is the product with the highest averaged spectral radius 
    among all products of length less or equal than three.
    Thus, $l_1=3$, $\rho_c=\rho(\Pi_1)^{1/l_1}=3^{1/3}$
    and we define 
    $\tilde{A}=\rho_c^{-1}A$, 
    $\tilde{B}=\rho_c^{-1}B$, 
    $\tilde{\mathcal{A}}=\{\tilde{A},\tilde{B}\}$,
    $\tilde{\Pi}_1=\tilde{B}\tilde{B}\tilde{A}$.
    \newline
    {($2$)} 
    The s.m.p.-candidate $\tilde{\Pi}_1$ has only one simple leading eigenvalue $1$ 
    with a corresponding eigenvector $v_1=v_1^{(0)}$ given by
    $v_1=v_1^{(0)}=5^{-1/2}[\,2\quad 1\,]^T$.
    \newline
    {($3$)} 
    We construct the cyclic root 
    $
    \mathcal{H}=
    \{v_1^{(0)},v_1^{(1)},v_1^{(2)}\}=\allowbreak
    \{v_1,\tilde{A}v_1, \tilde{B}\tilde{A}v_1\}=\allowbreak
    5^{-1/2}
    \big\{
    [\,1\quad 2\,]^T,\allowbreak
    [\,0\quad 3^{2/3}\,]^T,\allowbreak
    [\,3^{1/3}\quad 3^{1/3}\,]^T\allowbreak
    \big\}
    $        
    and set $V=\mathcal{H}$.
    \newline
    {($4$, first iteration) }
    We compute the norms of the vectors $\tilde{\mathcal{A}}V\setminus V$.
    The vector $\tilde{B}v_1$ is outside of the polytope $\operatorname{co}_- V$,
    $\|\tilde{B}v_1\|_{\operatorname{co}_- V}\simeq 1.04 \geq 1$,
    and thus, it is added to the set $V$.
    All other vectors, i.e.\ $\tilde{A}\tilde{A}v_1$ and 
    $\tilde{A}\tilde{B}\tilde{A}v_1$, 
    in the first iteration are inside of $\operatorname{co}_-  V\cup \tilde{B}v_1$;
    $\|\tilde{A}\tilde{A}v_1\|_{\operatorname{co}_- V\cup \tilde{B}v_1}\simeq 0.69<1$,
    $\|\tilde{A}\tilde{B}\tilde{A}v_1\|_{\operatorname{co}_- V\cup\tilde{B}v_1}\simeq 0.96<1$.
    \newline
    {($4$, second iteration)}
    We repeat step~\ref{alg_invpoly_testinside} and test the vectors from the set
    $\mathcal{A}(V\cup\tilde{B}v_1)\setminus (V\cup\tilde{B}v_1)$;
    $\|\tilde{B}\tilde{B}v_1\|_{\operatorname{co}_- V\cup \tilde{B}v_1}\simeq 0.92<1$,
    $\|\tilde{A}\tilde{B}v_1\|_{\operatorname{co}_- V\cup \tilde{B}v_1}\simeq 0.92<1$.
    \newline
    {($5$) }
    All vertices from the second iteration are mapped into the interior of the polytope
    $P=\operatorname{co}_- V\cup\tilde{B}v_1$, 
    therefore, $P$ is $\tilde{\mathcal{A}}$-invariant and $\JSR(\mathcal{A})=\rho(\Pi_1)^{1/l_1}=3^{1/3}\simeq 1.4422$.
}
\Description[Images showing the iterative construction of an invariant polytope.]{Images showing the iterative construction of an invariant polytope.}
\label{fig_invpoly}   
}     
\end{figure}

\section{Summary of the main modifications}\label{sec_outline}
In this section we present the modifications to the invariant polytope algorithm~\ref{alg_invpoly} 
and explain their importance. For more details see Sections~\ref{sec_modgrip} and~\ref{sec_modinvpoly}.

\subsection{New balancing procedure}
In steps~\eqref{alg_invpoly_leadingeigenvector} and \eqref{alg_invpoly_cyclicroot}
 of the explanation of the invariant polytope algorithm in~\ref{alg_invpoly},
 we only had one cyclic root, corresponding to the one leading eigenvector $v_1$.
If there happens to be more than one cyclic root, 
then it is necessary to balance the sizes of the cyclic roots to each other 
in order to ensure termination of the invariant polytope algorithm~\cite{GP16}.
There are (at the moment) three reasons why multiple cyclic roots occur:
$(1)$ The s.m.p.-candidate $\Pi_1$ possesses more than one leading eigenvalue,
or its leading eigenvalue is not simple,
or there are more than one s.m.p.-candidates $\Pi_1\ldots,\Pi_R$, $R\in\NN$.
But, multiple cyclic roots also can be generated by
$(2)$ artificially adding cyclic roots or by
$(3)$ artificially adding individual vertices.
Technique $(2)$ usually is employed whenever there are matrix products 
whose averaged spectral radius is \emph{nearly} that of the s.m.p.-candidate. 
Such matrix products are usually called \emph{nearly-s.m.p.s}~\cite[Remark~3.7]{GP16}.
If the leading eigenvectors of a nearly-s.m.p. are complex, 
one can take a real pair of vectors spanning the space generated by the complex eigenvectors.
Technique $(3)$ usually is employed whenever the initial polytope $\coast \mathcal{H}$ is very flat~\cite[Section 4]{GP16}.

The original balancing procedure described in~\cite[Sections~2.3 and 3]{GP16} 
may fail for multiple cyclic roots caused by the presence of nearly-s.m.p.s..
In Section~\ref{sec_modinvpoly_balancing} we improve on the original implementation 
such that it always works  and, in addition, automatically.
In Section~\ref{sec_modinvpoly_extravertex} we suggest an automated procedure 
how to select good nearly-s.m.p.s and extra vertices.
Aside from that, Example~\ref{ex_balancing_trans} presents a set of matrices
where it was wrongly assumed that no balancing is necessary.

\subsection{Finding s.m.p.\ candidates}
The invariant polytope algorithm~\ref{alg_invpoly} only terminates if all s.m.p.-candidates 
$\Pi_r$, $r=1,\ldots,R$, are indeed s.m.p.s.. 
Thus, the invariant polytope algorithm~\ref{alg_invpoly} 
heavily relies on correct initial guesses for the s.m.p.-candidates. 
A plain brute-force search in~\ref{alg_invpoly}~\eqref{alg_invpoly_cand} will fail, if the s.m.p.s length is large. 
Our numerical tests have shown that even for random pairs of $2\times 2$ 
matrices s.m.p.s of length greater than 30 are not uncommon. 
A particular easy example 
is given in Example~\ref{ex_longsmp}. 
We present two new methods that search for s.m.p.s efficiently in Sections~\ref{sec_modgrip} and~\ref{sec_modinvpoly_testrho}.

\subsection{Bounds for the JSR}
If the invariant polytope algorithm~\ref{alg_invpoly} does not find an invariant polytope in reasonable time,  
it can still give an upper bound for the $\JSR$ \emph{after} termination.
In Lemma~\ref{thm_termination} we show that our modified invariant polytope algorithm 
can return bounds for the $\JSR$ in \emph{each} iteration 
of the modified invariant polytope algorithm without the need of terminating the algorithm.

Nevertheless, these bounds are usually quite rough. 
A simple modification, presented in Remark~\ref{sec_modinvpoly_accuracy},
increases the accuracy of these intermediate bounds 
on the drawback that the exact value of the $\JSR$ becomes incomputable.

\subsection{Parallelization and natural selection of vertices}
A disadvantage of the invariant polytope algorithm~\ref{alg_invpoly} in its current form is 
that the polytope is changed \emph{inside} of the main loop in~\ref{alg_invpoly}~\eqref{alg_invpoly_testinside},  
which implies that in general the norm of $\tilde{A}_j v$ 
has to be computed with respect to a different polytope for each vertex. 
Therefore, the linear programming problem is different for each norm 
and the so-called warm start of linear programming problems cannot be used.
Furthermore, the main loop cannot be parallelised.
We eliminate these two drawbacks and additionally speed up the invariant polytope algorithm in Section~\ref{sec_modinvpoly_computenorm}.

The employed technique also solves a problem arising when the number of matrices in $\mathcal{A}$ is large.
In such cases the invariant polytope algorithm~\ref{alg_invpoly} will stall, 
simply due to the fact, that the number of vertices to test, 
increases in the worst case by a factor of $\#\mathcal{A}$ in each iteration.
E.g., if $\#\mathcal{A}\gtrsim 100$, 
the original invariant polytope algorithm is likely never to reach the third iteration.

\subsection{Estimating the Minkowski norm}
To reduce the number of norms one has to compute in~\ref{alg_invpoly}~\eqref{alg_invpoly_testinside}, 
we use the estimates for the Minkowski norm in Lemma~\ref{thm_estimate_1}.

\section{Modified Gripenberg algorithm}\label{sec_modgrip}

From Inequality~\eqref{equ_bound_jsr} we know that the normalized spectral radius of any matrix product is a lower bound for the $\JSR$. 
Thus, by a clever guess of a matrix product one easily obtains good (maybe sharp) lower bounds for the $\JSR$.  
Our new modified Gripenberg algorithm presented in this section finds in nearly all of our numerical tests an s.m.p..

The modified Gripenberg algorithm~\ref{alg_modgrip} is a modification of the well-known Gripenberg algorithm~\cite{Grip96}, 
one of the first algorithms which gave reasonable estimates for the $\JSR$. 
We briefly describe how it works:
Given some accuracy $0<\delta\leq 1$ we iteratively compute the sets $C_k$, $k\in\NN_0$. 
$C_0:=I$ and $C_{k+1}$ consists of all matrices ${C}\in\mathcal{A}C_k$ with $\|C\|^{1/(k+1)}\geq\delta^{-1} b_-$, 
where 
\begin{equation*}
b_-=\max\{\rho(C)^{1/n}:C\in C_n,\ n=1,\ldots,k\}
\end{equation*}
is the current lower bound for the $\JSR$.
In other words, we sort out matrix products whose averaged norm is less than the current lower  $b_-$ bound of the $\JSR$.
For each $k$ the $\JSR$ lies in the interval $[b_-,\ b_+]$ with 
\begin{equation*}
b_+=\min_{n=1,\ldots k} \max\big\{\|C\|^{1/n}:C\in C_n\big\}.
\end{equation*}
Note that $b_-$ is monotone increasing and $b_+$ is monotone decreasing.
If $\delta<1$ Gripenberg's algorithm terminates~\cite{Grip96}, 
i.e.\ there exists $K\in\NN$ such that $C_K=\emptyset$ and the $\JSR$ is computed up to an accuracy of $\delta$, i.e.\ $b_-/b_+\leq \delta$.
For real-world applications Gripenberg's algorithm works well for $\delta\leq0.95$.
For larger $\delta$ the number of products to compute is usually too large.
Figure~\ref{fig_grip} shows how to estimate the $\JSR$ using Gripenberg's algorithm for a concrete set of matrices.
For some vertex $w=A_j$, $v\in V_k$, $k\in\NN$, we say that $w$ is a \emph{child} of $v$, 
and that $v$ is the \emph{parent} of $w$.

\begin{figure}[tb]
{\small
\begin{forest} for tree={l sep=5}
[,baseline
    [{A\\[-.5ex]{\footnotesize $\norm{1.41}$}\\[-.5ex]{\footnotesize $\rho(1)$}\\[-.5ex]}
    [{AA\\[-.5ex]{\footnotesize $\norm{1.19}$}\\[-.5ex]{\footnotesize $\rho(1)$}\\[-.5ex]} 
    ]        
    [{BA\\[-.5ex]{\footnotesize $\norm{1.41}$}\\[-.5ex]{\footnotesize $\rho(1.41)$}\\[-.5ex]}
        [{{\phantom{ABA}\\[-.5ex]\phantom{\footnotesize $\norm{1.41}$}\\[-.5ex]\phantom{\footnotesize $\rho(1.26)$}\\[-.5ex]}}, no edge
        ]
        [{{\phantom{BBA}\\[-.5ex]\phantom{\footnotesize $\norm{1.47}$}\\[-.5ex]\phantom{\footnotesize $\rho(1.44)$}\\[-.5ex]}}, no edge
        ]
    ]
    ]            
    [{B\\[-.5ex]{\footnotesize $\norm{1.61}$}\\[-.5ex]{\footnotesize $\rho(1)$}\\[-.5ex]}
        [{AB\\[-.5ex]{\footnotesize $\norm{1.49}$}\\[-.5ex]{\footnotesize $\rho(1.41)$}\\[-.5ex]} 
            [{AAB\\[-.5ex]{\footnotesize $\norm{1.31}$}\\[-.5ex]{\footnotesize $\rho(1.26)$}\\[-.5ex]} 
            ]                     
            [{BAB\\[-.5ex]{\footnotesize $\norm{1.47}$}\\[-.5ex]{\footnotesize $\rho(1.44)$}\\[-.5ex]}
            ]        
        ]         
        [{BB\\[-.5ex]{\footnotesize $\norm{1.55}$}\\[-.5ex]{\footnotesize $\rho(1)$}\\[-.5ex]}
            [{ABB\\[-.5ex]{\footnotesize $\norm{1.47}$}\\[-.5ex]{\footnotesize $\rho(1.44)$}\\[-.5ex]} 
            ]            
            [{BBB\\[-.5ex]{\footnotesize $\norm{1.44}$}\\[-.5ex]{\footnotesize $\rho(1)$}\\[-.5ex]}
            ] 
        ]
    ]
]
\end{forest}
\hspace{3em}
\begin{forest} for tree={l sep=5}
[,baseline
    [{{\footnotesize $\phantom{\delta^{-1}}b_-=1\phantom{.00}$}\\[-.5ex]
      {\footnotesize $\phantom{\delta^{-1}}b_+ = 1.61$}\\[-.5ex]
      {\footnotesize $\delta^{-1} b_- = 1.05$}\\[-.5ex]}, no edge
        [{{\footnotesize $\phantom{\delta^{-1}}b_-=1.41$}\\[-.5ex]
          {\footnotesize $\phantom{\delta^{-1}}b_+ = 1.55$ }\\[-.5ex]
          {\footnotesize $\delta^{-1} b_- = 1.48$}\\[-.5ex]}, no edge
            [{{\footnotesize $\phantom{\delta^{-1}}b_-=1.44$}\\[-.5ex]
              {\footnotesize $\phantom{\delta^{-1}}b_+ = 1.47$ }\\[-.5ex]
              {\footnotesize $\delta^{-1} b_- = 1.51$}\\[-.5ex]}, no edge
            ]
        ]
    ]
]
\end{forest}
}
\caption[Tree built up by Gripenberg's algorithm]{%
    Tree built up by Gripenberg's algorithm for the matrices 
    $A=\left[\begin{smallmatrix}0&0\\1&1\end{smallmatrix}\right]$
    and
    $B=\left[\begin{smallmatrix}1&1\\0&1\end{smallmatrix}\right]$  
    with $\delta=0.95$ and using the $2$-norm.
    The computed matrices, their averaged norms and averaged spectral radii are printed.
    Gripenberg's algorithm terminates after the third iteration, since all extant matrices have averaged norm less then $\delta^{-1} b_-$.
    Thus, $\JSR\in[1.44,\, 1.47]$.
    More precise:
    \newline
    {\bf Iteration 1} Gripenberg's algorithm starts computing (averaged) norms and spectral radii of the matrices in the set $C_1=\mathcal{A}\cdot\{I\}=\{A,B\}$;
    $\norm{A}_2 \simeq 1.41$,
    $\norm{B}_2 \simeq 1.61$,
    $\rho(A) = 1$,
    $\rho(B) = 1$.
    Thus, we get the lower and upper bounds
    $b_-=\max\{\rho(A),\rho(B)\}=1$ and
    $b_+=\max\{\norm{A}_2,\norm{B}_2\}\simeq 1.61$ for the $\JSR$.
    The norms of both matrices is larger than $\delta^{-1}b_-$, thus, 
    $C_2=\mathcal{A}\{A,B\}=\{AA,BA,AB,BB\}$.
    \newline
    {\bf Iteration 2}
    Computing all averaged norms and spectral radii from the matrices in the set $C_2$, we obtain
    $b_-
    \simeq 1.41$,
    $b_+
    \simeq \allowbreak 1.55$. 
    Since,
    $\norm{AA}_2^{1/2},\norm{BA}_2^{1/2}<\delta^{-1}b_-$ we define 
    $C_3=\mathcal{A}\{AB,BB\}$.
    \newline
    {\bf Iteration 3} 
    Computing all averaged norms and spectral radii from the matrices in the set $C_3$, we obtain
    $b_-\simeq1.44$, $b_+\simeq1.45$.
    The averaged norms of all matrices in the set $C_3$ is less than $b_-$,
    and thus, $C_4=\emptyset$.
    The algorithm terminates and returns
    $\JSR(\mathcal{A})\simin[1.44,1.47]$.
    Note that, indeed, $\delta\leq 0.98\simeq1.44/1.47$.
}
\Description[Tree built up by Gripenberg's algorithm]{Tree built up by Gripenberg's algorithm}
\label{fig_grip}%
\end{figure}
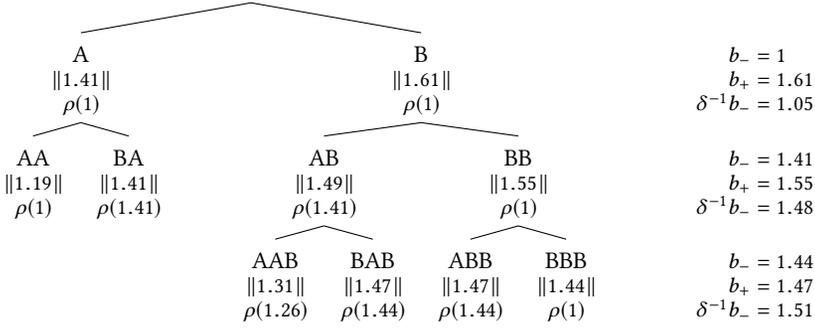

The modified Gripenberg algorithm~\ref{alg_modgrip} uses a different mechanism to sort out matrix products. 
Instead of just dismissing products with norms less than some threshold, 
it furthermore only keeps products with highest and lowest norms, 

\begin{figure}[t]
{\small
\begin{forest} for tree={l sep=5}
[,baseline
	[{A\\[-.5ex]{\footnotesize $\norm{1.41}$}\\[-.5ex]{\footnotesize $\rho(1)$}\\[-.5ex]}
    	[{AA\\[-.5ex]{\footnotesize $\norm{1.19}$}\\[-.5ex]{\footnotesize $\rho(1)$}\\[-.5ex]}
    	]
    	[{BA\\[-.5ex]{\footnotesize $\norm{1.41}$}\\[-.5ex]{\footnotesize $\rho(1.41)$}\\[-.5ex]}
        	[{ABA\\[-.5ex]{\footnotesize $\norm{1.41}$}\\[-.5ex]{\footnotesize $\rho(1.26)$}\\[-.5ex]}
        	]
        	[{BBA\\[-.5ex]{\footnotesize $\norm{1.47}$}\\[-.5ex]{\footnotesize $\rho(1.44)$}\\[-.5ex]}
        	]
    	]
	]
	[{B\\[-.5ex]{\footnotesize $\norm{1.61}$}\\[-.5ex]{\footnotesize $\rho(1)$}\\[-.5ex]}
    	[{AB\\[-.5ex]{\footnotesize $\norm{1.49}$}\\[-.5ex]{\footnotesize $\rho(1.41)$}\\[-.5ex]}
            [{{\phantom{AAB}\\[-.5ex]\phantom{\footnotesize $\norm{1.31}$}\\[-.5ex]\phantom{\footnotesize $\rho(1.26)$}\\[-.5ex]}} ,no edge
            ]
            [{{\phantom{BAB}\\[-.5ex]\phantom{\footnotesize $\norm{1.47}$}\\[-.5ex]\phantom{\footnotesize $\rho(1.44)$}\\[-.5ex]}},no edge
            ]
    	]
    	[{BB\\[-.5ex]{\footnotesize $\norm{1.55}$}\\[-.5ex]{\footnotesize $\rho(1)$}\\[-.5ex]}
            [{ABB\\[-.5ex]{\footnotesize $\norm{1.47}$}\\[-.5ex]{\footnotesize $\rho(1.44)$}\\[-.5ex]}
            ]
            [{BBB\\[-.5ex]{\footnotesize $\norm{1.44}$}\\[-.5ex]{\footnotesize $\rho(1)$}\\[-.5ex]}
            ]
    	]
	]
]
\end{forest}
\hspace{3em}
\begin{forest} for tree={l sep=5}
	[,baseline
	[{{\footnotesize $\rho_c=1\phantom{.00}$}\\[-.5ex]~\\[-.5ex]~\\[-.5ex]}, no edge
	[{{\footnotesize $\rho_c\simeq1.41$}\\[-.5ex]~\\[-.5ex]~\\[-.5ex]}, no edge
	[{{\footnotesize $\rho_c\simeq1.44$}\\[-.5ex]~\\[-.5ex]~\\[-.5ex]}, no edge
	]
	]
	]
	]
\end{forest}
}
\caption[Tree built up by the modified Gripenberg algorithm]{%
	Tree built up by the modified Gripenberg algorithm for the matrices 
	$A=\left[\begin{smallmatrix}0&0\\1&1\end{smallmatrix}\right]$
	and
	$B=\left[\begin{smallmatrix}1&1\\0&1\end{smallmatrix}\right]$     
	with $N=1$, $D=3$ and using the $2$-norm.
	The computed matrices, their averaged norms and averaged spectral radii are printed.
	The modified Gripenberg algorithm returns that $\JSR(\{A,B\})\gtrsim 1.44$.    
	More precise:
	\newline
	{\bf Iteration 1}
	We set $\mathcal{M}_0=\{I\}$, $\rho_c=0$.
	The modified Gripenberg algorithm starts computing averaged norms and spectral radii of the matrices in the set $\mathcal{M}_1=\mathcal{A}\mathcal{M}_0=\{A,B\}$;
	$\norm{A}_2\simeq \allowbreak 1.41$,
	$\norm{B}_2\simeq \allowbreak 1.61$,
	$\rho(A) = \allowbreak 1$,
	$\rho(B) = \allowbreak 1$.
	Thus, $\rho_c=\max\{0, \rho(A), \rho(B)\}=1$.
	Since $\norm{A}_2,\norm{B}_2\geq \rho_c$ no matrix products are removed,
	and $\mathcal{M}_1=\{A,B\}$.
	After sorting with respect to the (averaged) norms we obtain $\mathcal{M}_1=\{B,A\}$.
 Since $N=1$ we keep the first and last element of the sorted set,
	thus, $\mathcal{M}_1=\{B,A\}$.
	\newline
	{\bf Iteration 2}
	Computing the averaged norms and spectral radii in the set
	$\mathcal{M}_2=\mathcal{A}\mathcal{M}_1$.
    we obtain $\rho_c\simeq 1.41$.
	Since $\norm{AA}_2^{1/2} < \rho_c$ we set $\mathcal{M}_2=\{BA,AB,BB\}$.
	After sorting with respect to the averaged norms we obtain 
	$\mathcal{M}_2=\{BB,AB,BA\}$
	and since $N=1$ we keep the first and last element,
	thus, $\mathcal{M}_2=\{BB,BA\}$.
	\newline
	{\bf Iteration 3}
	Computing the averaged norms and spectral radii in the set
	$\mathcal{M}_3=\mathcal{A}\mathcal{M}_2$.
	we obtain 
    $\rho_c\simeq 1.44$.
	Since $D=3$ we stop in this iteration
	and return
	$\JSR(\mathcal{A})\gtrsim 1.44$
	and the set of s.m.p.-candidates $C=\{BBA\}$.
}
\Description[Tree built up by the modified Gripenberg algorithm]{Tree built up by the modified Gripenberg algorithm}
\label{fig_modgrip}
\end{figure}
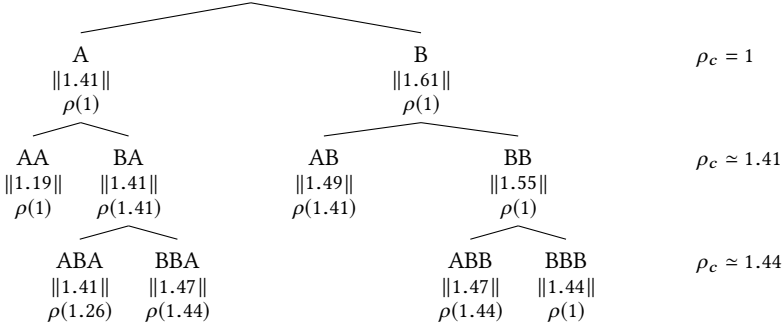

\begin{algorithm}[modified Gripenberg algorithm]\label{alg_modgrip}
\begingroup\allowdisplaybreaks
\begin{flalign*}
&\mathbf{Input: }\\
&\text{Set of square matrices }\mathcal{A}=\{A_j:j=1,\ldots,J\}\subseteq\RR^{s\times s}\\
&\text{Number of products kept in each step } N\in\NN\\
&\text{Maximal length of products } D\in\NN\\
&\mathbf{Output: }\\
&\text{S.m.p.-candidates $\mathcal{C}$}\\
&\text{Lower bound $\rho_c$ for $\JSR(\mathcal{A})$}\\[-0.6\baselineskip]
\cline{1-2}
&\mathbf{Initialization: }\\
&\text{Start with the product of length 0, }\mathcal{M}_0:=\{I\}, \text{ where } I \text{ is the identity matrix}\\
&\text{Set current lower bound for $\JSR$, } \rho_c:=0\\
&\mathbf{Algorithm: }\\
&\mathbf{for}\ d=1,\ldots,D\\
&\qquad \text{Compute all possible new products }\mathcal{M}_d:=\mathcal{A}\mathcal{M}_{d-1}\\
    &\qquad \text{Update lower bound } \rho_c:=\max\{\rho_c,\ \rho(M_d)^{1/d}:M_d\in\mathcal{M}_d\}\\
    &\qquad \text{Remove products whose norms are less than $\rho_c$, }
    \mathcal{M}_d:=\{M_d\in\mathcal{M}_d:\|M_d\|^{1/d}\geq \rho_c\}\\
    	&\qquad \numberthis\label{alg_modgrip_selectmatrix} \begin{aligned}
        &\text{Keep only products with highest and lowest norms:}\\
        &\quad \text{Sort } \mathcal{M}_d\ \text{w.r.t}\ \|M_d\| \text{ and sort out matrices with indices } N+1,\ldots, \#\mathcal{M}_d-N-1\\
        &\quad \text{Thus }\mathcal{M}_d=\{M_1,\ldots,M_N, M_{\#\mathcal{M}_d-N},\ldots,M_{\#\mathcal{M}_d} :M_i\in\mathcal{M}_d\}
    \end{aligned}\\
&\mathbf{Post\ processing:}\\
&\text{Choose products } \mathcal{C}=\{M_{d_i}\in\mathcal{M}_d : \rho(M_{d_i})^{1/d}=\rho_c,\ d=1,\ldots,D\}\\
&\text{Remove cyclic permutations and powers of  products from $\mathcal{C}$}\\
&\mathbf{return~} \mathcal{C},\ \rho_c
\end{flalign*}%
\endgroup
\end{algorithm}

\begin{theorem}\label{thm_modgrip}
The modified Gripenberg algorithm~\ref{alg_modgrip} has linear complexity in the 
number $J=\#\mathcal{A}$ of matrices , 
in the number $N\in\NN$ of kept products in each level  and 
in the maximal length  $D\in\NN$ of the products. 
\end{theorem}
\begin{proof}
	In every iteration, in total $D$ many, the modified Gripenberg algorithm computes at most $2\cdot N\cdot J$ norms and spectral radii.
\end{proof}

\begin{remark}
The modified Gripenberg algorithm~\ref{alg_modgrip} in the given form 
only returns lower bounds for the $\JSR$.
If one keeps track which products are dismissed, then it is possible to give also upper bounds for the $\JSR$.
Note that the modified Gripenberg algorithm with parameters 
$N=D=\infty$ is exactly Gripenberg's algorithm with accuracy $\delta=1$,
\end{remark}

\begin{remark}\label{rem_alg_randmodgrip}
Clearly one can pursue other selection strategies in step~\eqref{alg_modgrip_selectmatrix}.
The straightforward choice of taking the $2\cdot N$ products with highest normalized norm performs very badly.
Taking an arbitrary subset of $\mathcal{M}_d$ of size $2\cdot N$ in step~\ref{alg_modgrip}~\eqref{alg_modgrip_selectmatrix} 
performs mostly similarly to the modified Gripenberg algorithm~\ref{alg_modgrip}, 
but in some cases worse, see Table~\ref{table_random_modgrip} where we call it \emph{random Gripenberg algorithm}.
Furthermore, the modified Gripenberg algorithm~\ref{alg_modgrip} in the given form is deterministic, 
so we prefer it over a non-deterministic version.
\end{remark}

\begin{remark}
Our new modified invariant polytope algorithm, presented in Section~\ref{sec_modinvpoly}, 
can also be used to search for s.m.p.-candidates.
Thus, we give the numerical examples showing the performance of the modified Gripenberg algorithm~\ref{alg_modgrip}
only after Section~\ref{sec_modinvpoly}.
\end{remark}

Figure~\ref{fig_modgrip} shows how to find lower bounds for the $\JSR$ using the 
modified Gripenberg algorithm for a concrete set of matrices.
You may want to compare this Figure with Figure~\ref{fig_grip}.

\section{Modified invariant polytope algorithm}\label{sec_modinvpoly}

In this section, we present the modifications to the invariant polytope algorithm~\ref{alg_invpoly}.

\begin{algorithm}[Modified invariant polytope algorithm] \label{alg_modinvpoly}
Lines with numbers are subroutines,
described in detail in Sections~\ref{sec_modinvpoly_noinvsubspace}--\ref{sec_modinvpoly_testrho}.

\begingroup\allowdisplaybreaks
\begin{flalign*}
&\mathbf{Input: }\\
&\text{Set of irreducible square matrices }\mathcal{A}=\{A_j:j=1,\ldots,J\}\subseteq\RR^{s\times s}
\numberthis\label{alg_modinvpoly_noinvsubspace}
\\
&\text{Accuracy } 0<\delta\leq 1\quad (\delta\simeq 1)\\
&\text{Accuracy } 0<\epsilon<1  \text{ for computing the norms $N(v)$ in~\eqref{alg_modinvpoly_computenorm} }(\epsilon\simeq 0) \\
&\mathbf{Output: }\\
&\text{Exact value $\rho_c$ or bound $[\rho_c,\ b\cdot \rho_c]$ for $\JSR(\mathcal{A})$} \\
&\text{Invariant polytope~} \coast V\\
&\text{Spectral maximizing products~}\Pi_r&\\[-0.6\baselineskip]
\cline{1-2}
&\mathbf{Initialization: }\\
&
\text{Search for s.m.p.-candidates and nearly-s.m.p.s } \Pi_r=A_{j_{r_{l_r}}}\cdots A_{j_{r_1}},\ r=1,\ldots,R
\numberthis\label{alg_modinvpoly_modgrip} 
\\
&\text{Set } \rho_r:=\rho(\Pi_r)^{1/l_r},\ \rho_c:=\max \rho_r,\ \tilde{\mathcal{A}}:=\delta\rho_c^{-1}\mathcal{A}
\numberthis\label{alg_modinvpoly_accuracy}
\\
&\text{Compute the leading eigenvectors $v_r$ of $\tilde{\Pi}_r$}
\\
&\text{Compute the root vectors } v_r^{(i)}:=(\rho_c/\delta\rho_r)^i\tilde{A}_{j_{r_i}}\ldots \tilde{A}_{j_{r_1}}v_r,\ i=0,\ldots,l_r-1
\\
&
\text{Compute the extra-vertices } v_{R+1},\ldots,v_S\in\RR^s
\numberthis\label{alg_modinvpoly_extravertex}
 \\
&\text{Compute the balancing factors } \alpha_1,\ldots,\alpha_S\in\RR 
\numberthis\label{alg_modinvpoly_balancing} 
\\
&\text{Set } \mathcal{H}:=\{\alpha_1 v_1^{(0)},\alpha_1 v_1^{(1)},\alpha_1 v_1^{(2)},\ldots,\alpha_R v_R^{(l_R-1)}\},\ %
V_0:=\mathcal{H}\cup\{\alpha_{R+1} v_{R+1},\ldots,\alpha_S v_S\}\\
&\text{Set } N(v):=\infty \text{ for all } v\in V_0,\quad b_0:=\infty,\quad k:=0\\
&\mathbf{Main~Loop: }\\
&\mathbf{while~} \tilde{\mathcal{A}}V_k\setminus\mathcal{V}_k\not\subseteq (1-\epsilon)\coast  V_k\\
    &\qquad
    \text{Select new children } E_{k+1}\subseteq\tilde{\mathcal{A}}V_k\setminus\mathcal{V}_k 	
   	\text{ based on norm estimates}
    \numberthis\label{alg_modinvpoly_selectnewvertex} 
    \\
	&\qquad 
    \text{Choose subset of vertices } W_k\subseteq V_k		
    \numberthis\label{alg_modinvpoly_subset} 
    \\
	&\qquad 
    \text{Compute/classify norm~} N(v):=\|v\|_{\coast W_k} \text{ for all } v\in E_{k+1}
	\numberthis\label{alg_modinvpoly_computenorm}
    \\
	&\qquad V_{k+1}:=V_k\cup \{v\in E_{k+1}:N(v)>1-\epsilon\}
	\numberthis\label{alg_modinvpoly_addvertex} 
    \\
%
    &\qquad 
    b_{k+1}:=\min\big\{b_k,\ \max\{1,\ N(v)(1-\epsilon)^{-1}:v\in V_{k+1}\wedge \tilde{\mathcal{A}}v\nsubseteq V_{k+1}\}\}
\\
	&\qquad 
    \text{Test spectral radii based and eigenplane based stopping critera}
    \numberthis\label{alg_modinvpoly_testrho}
    \\
   	&\qquad \mathbf{print~} \JSR\in[\rho_c,\ \delta^{-1}\cdot b_{k+1}\cdot \rho_c]\\
	&\qquad k:=k+1\\		
&\mathbf{return~} V,\ \{\Pi_r\}_r,\ \rho_c
\end{flalign*}%
\endgroup%
\end{algorithm}


\begin{theorem}\label{thm_termination}
Let $\mathcal{A}=\{A_j:j=1,\ldots,J\}\subseteq\RR^{s\times s}$ be a finite set of square matrices.

	$(i)$ For $\delta=1$, the modified invariant polytope algorithm~\ref{alg_modinvpoly} terminates if and only if 
    the original invariant polytope algorithm~\ref{alg_invpoly} terminates, 
    i.e.\ $\Pi_1,\ldots,\Pi_R$ are dominant s.m.p.s and each s.m.p.\ possesses only one simple leading eigenvalue\footnote{In~\cite{GP13} such eigenvalues are called \emph{unique}.}.
    
	$(ii)$ For $0<\delta<1$ the modified invariant polytope algorithm~\ref{alg_modinvpoly} terminates if  
    $~\JSR(\mathcal{A})<\delta^{-1}\cdot \rho_c$.
    
	$(iii)$ Moreover, for any iteration $k\in\NN_0$, $\JSR(\mathcal{A})\in[\rho_c,\ \delta^{-1}\cdot b_{k+1}\cdot\rho_c]$, 
    where $\rho_c$ and and $b_{k+1}$ are defined in Algorithm~\ref{alg_modinvpoly}.
\end{theorem}

Before presenting the proof of Theorem~\ref{thm_termination} in Section~\ref{proof_termination}, 
we describe all modifications and extensions to the original invariant polytope algorithm~\ref{alg_invpoly}. 
These are numbered~\eqref{alg_modinvpoly_noinvsubspace}--\eqref{alg_modinvpoly_testrho} 
in the modified invariant polytope algorithm~\ref{alg_modinvpoly}.
All heuristic constants which influence the behaviour of the algorithm can be changed by passing a 
name-value pair in the function call of our implementation, see the documentation for more information.

\subsection{Irreducibility of input matrices~\texorpdfstring{\eqref{alg_modinvpoly_noinvsubspace}}{}}
\label{sec_modinvpoly_noinvsubspace}
The set of matrices $\mathcal{A}$ should be irreducible, 
i.e.\ the matrices in the set $\mathcal{A}$ should not have a trivial common invariant subspace,
because otherwise (both the modified~\ref{alg_modinvpoly} and) the invariant polytope algorithm~\ref{alg_invpoly} may not be able to terminate. 
If the matrices are reducible, then there exists a basis in which all of the matrices $A_j$ have block upper triangular form. 
The $\JSR$ of the matrices then equals to the maximum of the $\JSR$ of the diagonal blocks.
In our implementation
we therefore automatically search for non-trivial common invariant subspaces prior to starting the modified invariant polytope algorithm.
Here we make use of the functions~\texttt{permTriangul} and~\texttt{jointTriangul} from~\cite{Jung2014}, 
as well as a new method \texttt{invariantsubspace}
which searches for non-trivial common invariant difference subspaces as described in~\cite{CP17}.

\subsection{Search for s.m.p.-candidates~\texorpdfstring{\eqref{alg_modinvpoly_modgrip}}{}}
\label{sec_modinvpoly_modgrip}
We use the modified Gripenberg algorithm~\ref{alg_modgrip} to search for s.m.p.-candidates and nearly-s.m.p.s.
Every product, which is shorter than the s.m.p.-candidate and having normalized spectral radius greater or equal to $\tau\cdot\rho_c$ is considered to be a nearly-s.m.p.. 
In our implementation we use a heuristic default value of $\tau=0.9999$
and use the Matlab function \texttt{eig} to compute the leading eigenvalues. 
This may not be the fastest available procedure, 
but it is fast enough in comparison to the time the main loop needs to terminate.

\subsection{Approximate computation~\texorpdfstring{\eqref{alg_modinvpoly_accuracy}}{}}
\label{sec_modinvpoly_accuracy}
If we multiply the set of matrices $\tilde{\mathcal{A}}$ by a factor $0<\delta<1$, 
the modified invariant polytope algorithm~\ref{alg_modinvpoly} cannot return exact values for the $\JSR$ anymore, 
but only up to a relative accuracy of $\delta$. 
Indeed, if the modified invariant polytope algorithm~\ref{alg_modinvpoly} terminates, then 
$\|\tilde{A}_j v\|_{\coast V}\leq 1 \Leftrightarrow
\|A_j v\|_{\coast V}\leq\delta^{-1}\cdot\rho_c\Leftrightarrow
\JSR(\mathcal{A})\leq\delta^{-1}\cdot\rho_c$.
There are cases where this procedure is of significance.

\paragraph{(a)}
 If the dimension $s$ of matrices is large, 
 (both the modified~\ref{alg_modinvpoly} and)
 the invariant polytope algorithm~\ref{alg_invpoly}, will probably not terminate anyway, 
 and thus only give bounds for the $\JSR$. 
 A factor $\delta\simeq 0.97$ will speed up the computation tremendously and 
 the returned bounds from the modified invariant polytope algorithm are mostly better 
 (at least in our numerical examples) than for \mbox{$\delta=1$}.
 The value $0.97$ is based on numerical experiments. 
 An optimal value for $\delta$ can probably be determined using the spectral gap at $1$, 
 but no theoretical investigations nor numerical experiments in that direction have been taken so far.
 
\paragraph{(b)}
 If the s.m.p.s are not dominant, or there is an infinite number of dominant s.m.p.s, 
 or $\mathcal{A}$ is not irreducible,
 the modified invariant polytope algorithm~\ref{alg_modinvpoly} will not terminate. 
 In these cases, choosing $\delta\simeq1-10^{-9}$ ensures that the modified invariant polytope algorithm~\ref{alg_modinvpoly} terminates and 
 the obtained bounds will be nearly the same as when $\delta=1$.
 Note that these cases are mostly non-generic, 
 except for matrix families where this property is known to hold a priori, 
 for example certain matrix sets occurring in subdivision.

\paragraph{(c)}
 If one is interested only whether $\JSR(\mathcal{A})< B$ for some $B>0$, one can choose $1>\delta>B^{-1}\rho_c$ and  the modified invariant polytope algorithm~\ref{alg_modinvpoly} will terminate much faster.

\subsection{Adding extra-vertices automatically~\texorpdfstring{\eqref{alg_modinvpoly_extravertex}}{}}
\label{sec_modinvpoly_extravertex}
The aim of this step in the algorithm is, to compute vertices 
$E=\{v_{R+1},\allowbreak\ldots,\allowbreak v_S\}$
such that the polytope $\coast (\mathcal{H}\cup E)$, $\mathcal{H}=[v_1^{(0)},\allowbreak v_1^{(1)},\allowbreak \ldots,\allowbreak v_R^{(l_R-1)}]$, has non-empty interior and is elongated in all coordinate directions.
The procedure is different in cases $(P)$ and $(R)$.

For case $(R)$, 
given some threshold $T>0$, 
we compute the singular value decomposition of 
$\mathcal{H}=[v_1^{(0)},\allowbreak v_1^{(1)},\allowbreak \ldots,\allowbreak v_R^{(l_R-1)}]$.
and take all singular vectors (which thus become \emph{extra-vertices}) 
$E=\{v_{R+1},\allowbreak\ldots,\allowbreak v_S\}$ corresponding to singular values
which are in modulus less than $T$.
Note that the singular vectors form an orthonormal system
and that the singular vectors corresponding to small singular values are exactly the directions in which the polytope $\coast \mathcal{H}$ has small or even no elongation. 
In particular, the polytope $\coast (\mathcal{H}\cup E)$ has always non-empty interior.

For case $(P)$, $e_n\in E$ whenever $v_n\leq T$ for all $v\in \mathcal{H}$, 
where $e_n$ is the $n^{th}$ unit vector of $\RR^s$,

In our implementation we use a heuristic value of $T\simeq 0.1$ for both cases.

\subsection{Balancing of cyclic trees~\texorpdfstring{\eqref{alg_modinvpoly_balancing}}{}}
\label{sec_modinvpoly_balancing}
As already noted, the existence of multiple cyclic roots makes it necessary to 
balance the sizes of the cyclic roots to each other in order that the invariant polytope algorithm can terminate.
The balancing procedure uses the 
\emph{dual leading eigenvectors} $v_r^*$, $r=1,\ldots,R$.
More precisely,
for the s.m.p.-candidate $\tilde{\Pi}_r$ define $\tilde{\Pi}_r^*v_r^*=v_r^*$ with 
$\langle v_r^{(0)},v_r^*\rangle=1$,
where $\Pi_r^*$ is the conjugate transpose of $\Pi_r$ and 
$\langle\,\cdot\,,\cdot\,\rangle$ is the standard inner product~\cite[Section 2.3]{GP16}.

If $\delta<1$ no balancing is necessary by Theorem~\ref{thm_termination}. 
If $\delta=1$ we define for $h\in\NN$
\begin{align*}
\left\{
\begin{aligned}
q_{i,j}&=\sup_{z\in\tilde{\mathcal{A}}^h\{(\rho_c/\rho_r)^iv_i^{(0)},\ldots,v_i^{(l_i-1)}\}} |\langle v_j^*,z\rangle|, &i=1,\ldots,R\phantom{+1}\\
q_{i,j}&=\sup_{z\in\tilde{\mathcal{A}}^hv_i} |\langle v_j^*,z\rangle|, &i=R+1,\ldots,S
\end{aligned}
\right.
\quad,\ j=1,\ldots,R.
\end{align*}
The factor $(\rho_c/\rho_r)^i$ ensures that all vertices of the cyclic root of nearly-s.m.p.s 
get the same weight in the computation. 
If $v_i$ is the leading eigenvector of an s.m.p.-candidate, we have $\rho_c/\rho_r=1$.
Now one has to find numbers $\alpha_1,\ldots,\alpha_S>0$ such that
\begin{equation*}
\left\{
\begin{array}{rcll}
\alpha_i q_{i,j}&<&\alpha_j  &\text{whenever $v_i$ is the leading eigenvector of an s.m.p.-candidate}\\
\alpha_i q_{i,j}&<&1         &\text{otherwise}
\end{array}
\right.
\end{equation*}
and multiply all vertices $v_i^{(j)}$, $i=1,\ldots,R$, $j(i)=0,\ldots,l_i-1$, 
and extra-vertices $v_i$, $i=R+1,\ldots,S$, 
from the root $\mathcal{H}$ with the corresponding balancing factor $\alpha_i$. 
In our implementation we distinguish between extra-vertices and vertices from nearly-s.m.p.s., 
precisely we solve the following system
\begin{equation*}
\left\{
\begin{array}{rcll}
\alpha_i q_{i,j}&<&\alpha_j             &\text{whenever $v_i$ is the leading eigenvector of an s.m.p.-candidate}\\
\alpha_i q_{i,j}&=&B_{\text{nearly}}\cdot \rho_i     &\text{whenever $v_i$ is the leading eigenvector of a nearly-s.m.p.}\\
\alpha_i q_{i,j}&=&B_{\text{extra}}                &\text{whenever $v_i$ is an extra-vertex}
\end{array}
\right.,
\end{equation*}
where $B_{\text{nearly}}=0.999$ and $B_{\text{extra}}=0.01$ are based on numerical experiments.
\cite[Theorem~3.3]{GP16} ensures that the modified invariant polytope algorithm~\ref{alg_modinvpoly} 
terminates when started with both the balanced s.m.p.-candidates, 
nearly-s.m.p.s and extra-vertices if and only if 
it terminates when started solely with the balanced s.m.p.-candidates.

\noindent
It was assumed (personal communication), at least for dimension $s=1$, 
that the balancing factors for transition matrices occurring in subdivision theory\footnote{
Subdivision schemes are computational means for generating finer and finer meshes in $\RR^s$, 
usually in dimension $s=1,2,3$. At each step of the subdivision recursion, 
the topology of the finer mesh is inherited from the coarser mesh and the coordinates 
$c^{(n+1)}$ of the finer vertices are computed by local averages of the coarser ones $c^{(n)}$ by 
$c^{(n+1)}=Sc^{(n)}=\sum_{\alpha\in\ZZ^s} a(\cdot-M\alpha)c^{(n)}(\alpha)$.
See~\cite{CM18} for a more thorough explanation.
} are always equal to $1$.
While it is not hard to find counterexamples in dimensions $s>1$, 
the claim is also not valid in the univariate case, as Example~\ref{ex_balancing_trans} shows.
Readers unfamiliar with subdivision schemes may skip Example~\ref{ex_balancing_trans}.
\begin{example}\label{ex_balancing_trans}
Let $S$ be the univariate subdivision scheme defined by the mask $a$ and the dilation matrix $M$ given by
{
\begin{equation*}
a=
\frac{1}{12}
[\!\!\begin{array}{ccccccccccc}
3&3&4&3&3&4&3&3&4&3&3
\end{array}\!\!]^T
,\quad
M=
-3
.
\end{equation*}
}

The basic limit function can be seen in Figure~\ref{fig_balancing}.
Taking the digit set $D=\{-2,\ -1,\ 0\}=M[0, 1)\cap\ZZ$,  we construct the set
$\Omega_C=\{-4,\ -3,\ -2,\ -1,\ 0,\ 1\}$ (using~\cite[Lemma~3.8]{CM18}) and the corresponding transition matrices 
$T_d=\big[a(\alpha-M\beta)\big]_{\alpha,\beta\in\Omega_C}$, $d\in D$. 
The restriction of the transition matrices to the space $V$ of first order differences with basis
{\small
\begin{gather*}
\left[\!\begin{array}{rrrrr}
     1  &   0 &    0  &   0 &    0\\
    -1  &   1 &    0  &   0 &    0\\
     0  &  -1 &    1  &   0 &    0\\
     0  &   0 &   -1  &   1 &    0\\
     0  &   0 &    0  &  -1 &    1\\
     0  &   0 &    0  &   0 &   -1\\
\end{array}\right]
\end{gather*}
}
yields the set of matrices $\mathcal{T}|_V=\{T_{-2}|_V,\ T_{-1}|_V,\ T_{0}|_V\}$ with
{\small%
\begin{gather*}
T_{-2}|_V\!=
\!\frac{-1}{12}\!
\left[\!\!\begin{array}{rrrrr}
    0   &  0  &  0  &  3  &   0\\
    3   &  0  &  1  &  2  &   0\\
    2   &  0  &  2  &  1  &   0\\
    1   &  0  &  3  &  0  &   0\\
    0   &  0  &  0  &  0  &   0\\
\end{array}\!\!\right],\
T_{-1}|_V\!=
\!\frac{-1}{12}\!
\left[\!\begin{array}{rrrrr}
    0  &  0   &  0   & 0  &  3\\
    0  &  3   &  0   & 1  &  2\\
    1  &  2   &  0   & 2  &  1\\
    2  &  1   &  0   & 3  &  0\\
    3  &  0   &  0   & 0  &  0\\
\end{array}\!\right],\
T_{0}|_V\!=
\!\frac{-1}{12}\!
\left[\!\begin{array}{rrrrr}
     0  &  0  &  0  &   0  &  0\\
     0  &  0  &  3  &   0  &  1\\
     0  &  1  &  2  &   0  &  2\\
     0  &  2  &  1  &   0  &  3\\
     0  &  3  &  0  &   0  &  0\\
\end{array}\!\right].
\end{gather*}
}%
For the s.m.p.s $\Pi_1=T_{-2}T_{-1}T_{-1}|_V$ and $\Pi_2=T_{-1}T_{-1}T_{0}|_V$ with balancing vector 
$[\!\!\begin{array}{cc}1&9/10\end{array}\!\!]$,
the original invariant polytope algorithm terminates after 4 iterations.
Without bal\-ancing the original invariant polytope algorithm does not terminate.
\end{example}

\noindent
Example~\ref{ex_balancing_nearly} shows the advantage of the new balancing procedure in connection with nearly-s.m.p.s.
\begin{example}\label{ex_balancing_nearly}
Given $E_1=\left[\!\!\begin{array}{r r}2 &1\\-1&2\end{array}\!\!\right]$,
$E_2=\left[\!\!\begin{array}{c c}2 &0\\2& 1\end{array}\!\!\right]$,
the irreducible set $\mathcal{E}=\{E_1,E_2\}$ has $E_2 E_1$ as an s.m.p. and $\rho(\mathcal{E})=2.5396\ldots$.
Assuming we start the modified invariant polytope algorithm~\ref{alg_modinvpoly} 
with that candidate and the nearly-s.m.p. $E_2$, with corresponding leading eigenvectors 
$v_1^{(0)}=[\!\!\begin{array}{c c}0.9121\ldots &\!\!0.4100\ldots\end{array}\!\!]^T$,
$v_2^{(0)}=[\!\!\begin{array}{c c}0.4472\ldots &\!\!0.8944\ldots\end{array}\!\!]^T$ and leading dual eigenvectors
$v_1^*=    [\!\!\begin{array}{c c}0.9958\ldots &\!\!0.2238\ldots\end{array}\!\!]^T$,
$v_2^*=    [\!\!\begin{array}{c c}2.2361\ldots &\!\!0.0000\ldots\end{array}\!\!]^T$. 
For the balancing procedure as described in~\cite[Remark~3.7]{GP16}
we need to find numbers $\alpha_1,\alpha_2>0$ such that for some $h\in\NN$, say $h=10$, 
$q_{1,2}=\sup_{z\in\tilde{\mathcal{E}}^h\{v^{(0)}_1, v^{(1)}_1\}} |(v_2^*,z)|=2.0395\ldots$ and
$q_{2,1}=\sup_{z\in\tilde{\mathcal{E}}^h\{v^{(0)}_2\}} |(v_1^*,z)|=0.8196\ldots$ the following two inequalities hold
\begin{align*}
\alpha_1 \cdot 2.0395\ldots = \alpha_1 q_{1,2}&<\alpha_2\\
\alpha_2 \cdot 0.8196\ldots =  \alpha_2 q_{2,1}&<\alpha_1.
\end{align*}
This is clearly impossible since $2.0\times 0.8>1$.
Because there are no admissible balancing factors for $h=10$, 
there are no admissible balancing factors for $h>10$~\cite[Section 3]{GP16}. 

Since $E_1E_2$ is a dominant s.m.p., 
the modified invariant polytope algorithm~\ref{alg_modinvpoly} terminates 
if it is started only with that candidate, 
and thus, there exist balancing factors such that the the modified invariant polytope algorithm terminates 
when started with $E_2E_1$ and $E_1$, e.g.\ $\alpha_1=1$, $\alpha_2\simeq 0.95$ as given by our new method.

\end{example}

\begin{figure}[t]
	\centering
	\includegraphics{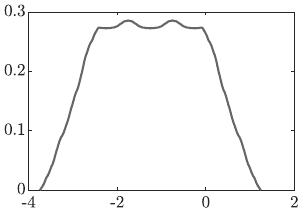}	
	\caption{The basic limit function for the subdivision scheme from Example~\ref{ex_balancing_trans}.}
    \Description[The function has support roughly -3.6 to 1.3 and is strictly positive.]%
    {The function has support roughly -3.6 to 1.3 and is strictly positive.}
	\label{fig_balancing}
\end{figure}

\subsection{Select new children -- Natural selection of vertices~\texorpdfstring{\eqref{alg_modinvpoly_selectnewvertex}}{}}
\label{sec_modinvpoly_selectnewvertex}
In the original invariant polytope algorithm~\ref{alg_invpoly}, in every iteration all vertices generated in the last iteration, 
which were not mapped inside the polytope, were used to construct new vertices. 
In the modified invariant polytope algorithm~\ref{alg_modinvpoly} we only take a subset of those. 
We choose the vertices under the mild condition that 
\begin{equation}\label{equ_modinvpoly_selectnewvertex}
\text{
for every $n\in\NN$, every vertex of $\{\tilde{A}_j\}^n V_0$ eventually will be selected,
}
\end{equation}
given that it is not absorbed already. 
In other words, we do not forget any vertex to select. 
This condition is necessary to proof that the modified invariant polytope algorithm and the original invariant polytope algorithm
have the same qualitative behaviour in Theorem~\ref{thm_termination}.

Two selection strategies turned out to work well:
\begin{itemize}
\item[$(a)$]
Choose those vertices that have the largest (e.g.\ highest decile) norm $\|V_k^+\vardot\|_2$, 
where $V_k^+$ denotes any pseudo-inverse of $V_k$.
In view of Lemma~\ref{thm_estimate_1}~$(2)$, the value $\|V_k^+ v\|_2$ is an approximation of $\|v\|_{\coast V_{k}}$
and, thus, we may assume that vertices $v$ with high value $\|V_k^+ v\|_2$ are far outside of the polytope~$\coast V_k$.
\item[$(b)$]
Choose those vertices whose parent vertex has largest norm with respect to the norm $\|\vardot\|_{\coast V}$.
\end{itemize}

With a good selection of new vertices, the polytope $\coast V_k$ gets large faster, 
thus, can absorb new vertices faster, and so the number of vertices of the invariant polytope may be smaller.
Strategy~$(a)$ reduces the number of vertices of the invariant polytope by roughly 20\%, 
strategy~$(b)$ by roughly 10\%. 
Since the intermediate bounds $b_k$ for the $\JSR$ decreases very slowly when we use strategy~$(a)$ only, 
we use three times~$(a)$ and one time~$(b)$ in our implementation.

\begin{algorithm}[{Subroutine \emph{Natural selection of vertices \eqref{alg_modinvpoly_selectnewvertex}}}]
\begingroup\allowdisplaybreaks%
\begin{flalign*}
&\mathbf{Input}\ V_k,\ \mathbf{Output}\ E_{k+1}&\\[-0.6\baselineskip]
\cline{1-2}
&\mathbf{if}\ k \neq 0 \text{ mod } 4 \ \mathbf{then}\ %
\text{compute } y_v=\|V_k^+ v\| \text{ for all } v\in \mathcal{A}V_k\setminus\mathcal{V}_k\\
&\mathbf{else}\
\text{set}\ y_v=\|w\|_{E_k} \text{ for all } v\in\mathcal{A}w\setminus\mathcal{V}_k,\ w\in\mathcal{V}_{k}\\
&\text{sort } E_k\\
&E_{k+1}:=
\text{%
Choose 10\%,
but at least $4\cdot\#\text{thread}$,
of the highest values in $E_k$\newline
and such that \eqref{equ_modinvpoly_selectnewvertex} holds
}
\end{flalign*}%
\endgroup%
\end{algorithm}
In Algorithm~\ref{alg_modinvpoly_selectnewvertex}, 
we denote with $\#\text{thread}$ the number of available threads of the computer.
The natural selection of new vertices also makes the modified invariant polytope  algorithm~\ref{alg_modinvpoly} 
applicable for problems with a large number of matrices, 
since it ensures that the number of norms to be computed in each iteration is reasonably small. 

\subsection{Simplified polytope~\texorpdfstring{\eqref{alg_modinvpoly_subset}}{}}
\label{sec_modinvpoly_subset}

In each iteration $k$ we take a subset $W_k\subseteq V_k$ of vertices 
which are used to compute the norms in step~\eqref{alg_modinvpoly_computenorm} 
for the vertices in $E_{k+1}$ due to 2 reasons.

Firstly, in some examples the vertices constructed by the modified invariant polytope algorithm~\ref{alg_modinvpoly} are very near to each other, i.e.\ are at distances in the order of the machine epsilon. 
Those vertices are irrelevant for the size of the polytope and so we disregard them.
This also protects against stability problems in the LP-programming part, since for simplices with vertices very near to each other, LP-solvers perform very badly.
This phenomenon happens frequently when there are multiple s.m.p.s..
In our implementation we use a variable threshold in~\eqref{alg_modinvpoly_subset} when determining which vertices of the polytope we use in the computation of the norm.

Secondly, as we will see in the proof of Theorem~\ref{thm_termination}, in order to obtain intermediate bounds $b_{k+1}$ for the $\JSR$, 
we are only allowed to choose vertices whose children are selected for its norm to be computed,
or whose children norms are already computed, i.e.\ it must be satisfied that
\begin{equation}\label{equ_Wk}
\tilde{\mathcal{A}}W_k\subseteq V_k\cup E_{k+1}.
\end{equation}

It would also be possible to choose a polytope $W(v)$ for each norm $\|v\|_{\coast W(v)}$ we need to compute,
since for each $v\in\RR^s$ we only need $s+1$ vectors from $V$ to compute the norm $\|v\|_{\coast V}$ exactly. 
Unfortunately we have no idea so far, how to select a good subset of $V_k$ in a reasonable amount of time, 
i.e.\ faster than the computation of the norm would take.

\subsection{Parallelisation~\texorpdfstring{\eqref{alg_modinvpoly_computenorm}}{}~\&~\texorpdfstring{\eqref{alg_modinvpoly_addvertex}}{}}
\label{sec_modinvpoly_computenorm}
\label{sec_modinvpoly_addvertex}
This is one of the main differences to the original implementation -- the idea is already developed in~\cite[Algorithm~5.1]{GZ08}.
Instead of testing each vertex one after another, and adding it immediately to the set of vertices $V_k$ if it is outside of the polytope, 
we compute the norms of all selected vertices from step~\eqref{alg_modinvpoly_selectnewvertex} with respect to the same polytope. 
Afterwards we add all vertices which are outside of the polytope at once to the set $V_k$.

This clearly leads to larger polytopes, in our examples the number of vertices increases by~10\%, 
but this is compensated by the fact that we can parallelise the computations of the norms. 
The speed-up is nearly linear in the number of available threads.
Since the linear programming model does not change, we can speed up this part further by warm starting the linear programming problems, 
i.e.\ we reuse the solutions obtained from the computations of the other vertices. 
If there are no suitable candidates to warm start with, 
we still can speed up the LP-problem by starting the search for the solution at the nearest vertex point of the polytope $W$.
The speed-up from warm starting is roughly 50-70\%.


\subsection{Norm classification~\texorpdfstring{\eqref{alg_modinvpoly_computenorm}}{}}
Before computing the exact norm of a vector $A_jv$, we try to determine the relative position 
(inside or outside of the polytope) using the estimates in Lemma~\ref{thm_estimate_1}. 
If a vertex is proven to be inside or outside of the polytope, we do not have to compute its exact norm anymore. 
Unfortunately, these estimates are quite rough and fail to determine the position for most vertices, 
except in case $(P)$ where Lemma~\ref{thm_estimate_1}~\eqref{thm_estimate_1_pos} gives very good estimates.


\begin{algorithm}[{Subroutine \emph{Norm classification~ and adding of vertices}~\eqref{alg_modinvpoly_computenorm}~\eqref{alg_modinvpoly_addvertex}}]
	\begingroup\allowdisplaybreaks%
	\begin{flalign*}
	&\mathbf{Input}\ E_{k+1},\ \mathbf{Output}\ V_{k+1}&\\[-0.6\baselineskip]
     \cline{1-2}
	&V_{k+1}:=V_k\\
	&\mathbf{for}\ v\in E_{k+1}\\
	&\quad\text{Classify } \|v\|_{\coast W_k} \text{ using Lemma~\ref{thm_estimate_1}}\\
	&\quad\mathbf{If}\ \text{$v$ is outside of $\coast V_k$}
	\ \mathbf{then}\ N(v):=\infty\\
	&\quad\mathbf{else\ if}\ \text{$v$ is inside of $\coast V_k$}\ \mathbf{then}\ N(v):=0\\
	&\quad\mathbf{else}\ N(v):=\|v\|_{\coast W_k}\\
	&\quad\mathbf{If}\ N(v)>1-\epsilon\ \mathbf{then}\ V_{k+1}=V_{k+1}\cup v 
	\end{flalign*}%
	\endgroup%
\end{algorithm}

\subsection{Spectral radius based stopping criterion~\texorpdfstring{\eqref{alg_modinvpoly_testrho}}{}}
\label{sec_modinvpoly_testrho}
The spectral radius based stopping criterion is used to find better s.m.p.-candidates, in case the chosen s.m.p.-candidates $\Pi_r$ are no s.m.p.s.
If the s.m.p.-candidates $\Pi_r$ are s.m.p.s,
then all intermediately occurring matrix products  will have spectral radius less than $1$. 
Unfortunately, the converse is not true.

\begin{algorithm}[{Subroutine \emph{Spectral radius based stopping criterion}~\eqref{alg_modinvpoly_testrho}}]
    \begin{flalign*}
    &\mathbf{Input}\ v\in V_{k+1},\ \mathbf{Output}\text{ Maybe a better s.m.p.-candidate.}&\\[-0.6\baselineskip]
    \cline{1-2}
    &\mathbf{for}\ v=\tilde{A}_{j_n}\cdots \tilde{A}_{j_0} v_s^{(0)}\in V_{k+1}\ \\
    &\qquad\text{Compute } \rho = \rho(\tilde{A}_{j_n}\cdots \tilde{A}_{j_0})^{1/n}\\
    &\qquad\mathbf{if}\ \rho>1 \ \mathbf{then} 
    \text{ restart algorithm with s.m.p.-candidate ${A}_{j_n}\cdots {A}_{j_0}$}
    \end{flalign*}%
\end{algorithm}

As noted, if the candidates are not s.m.p.s, 
it can happen that all intermediately occurring matrix products have spectral radius less than $1$ 
and that the modified invariant polytope algorithm~\ref{alg_modinvpoly} never stops, 
see Example~\ref{ex_testrho}.
Nevertheless, this never happened in any non-artificial example. 
Furthermore, products with larger normalized spectral radius always occurred very fast. 
Thus, from a practical point of view, 
the spectral radius based stopping criterion is a better way to check whether the candidates are s.m.p.s 
than the eigenplane based method described in~\cite[Proposition 2]{GP13}. 
On the other hand, whenever the eigenplane based method~\cite[Proposition 2]{GP13} is applicable,
it is fail-proof and eventually will strike if an s.m.p.-candidate is not an s.m.p..
Thus, in our implementation of the modified invariant polytope algorithm both stopping criteria are used.

We now illustrate how the new stopping criterion~\eqref{alg_modinvpoly_testrho} may fail. 
For that purpose, we introduce for given $\eta\geq0$ the set 
\begin{equation*}
\mathcal{M}_\eta=\{
(j_n)_{n}\in\{1,\ldots,J\}^\NN\ :\ %
\rho(\tilde{A}_{j_m}\cdots \tilde{A}_{j_1})^{1/m}\leq\eta,\quad \forall m\in\NN
\}.
\end{equation*}
For $\eta=1$, the products $\tilde{A}_{j_n}\cdots \tilde{A}_{j_1}$, $(j_n)_n\in \mathcal{M}_1$, 
are exactly the products occurring in the modified invariant polytope algorithm~\ref{alg_modinvpoly} until the spectral radius based stopping criterion~\eqref{alg_modinvpoly_testrho} strikes.
The hope would be, that the norms of the products in that sequence stay bounded, 
i.e.\ $\exists\,C>0$ such that $\|\tilde{A}_{j_n}\cdots \tilde{A}_{j_1}\|<C$ for all $n\in\NN$. 

 \begin{example}\label{ex_testrho}
Let
$A=\left[\begin{array}{c c}1 & 1\\0&1\end{array}\right]$,
$B=\left[\begin{array}{c c}1 & 0\\1&1\end{array}\right]$.
Clearly $\JSR(\{A,B\})=\rho(AB)^{1/2}=(\sqrt{5}+1)/2$ and $\{A,B\}$ is irreducible.
We choose $\Pi_1=A$ and $\Pi_2=B$ as our two (wrong) s.m.p.-candidates, thus, 
$\rho_c=\rho_r=1$,
$\tilde{\Pi}_1=\tilde{A}=A$, $\tilde{\Pi}_2=\tilde{B}=B$,
$V_0=\mathcal{H}=\{v_1^{(0)},v_2^{(0)}\}$ with
$v_1^{(0)}=[\!\!\begin{array}{c c}1&0\end{array}\!\!]^T$,
$v_2^{(0)}=[\!\!\begin{array}{c c}0&1\end{array}\!\!]^T$,
and since there are no extra vertices, $V_0=\mathcal{H}$.

Now, we use a (bad) selection procedure of new vertices $E_{k+1}$ in~\eqref{alg_modinvpoly_selectnewvertex} 
of the natural selection of vertices;
namely, we choose only the vertices $A^n v_2$ and $B^n v_1$, $n\in\NN$.\footnote{
Actually, this selection of vertices is neither type $(a)$ or $(b)$ from Section~\ref{sec_modinvpoly_selectnewvertex},
nor does it fulfil the necessary condition~\eqref{equ_Wk}.
}

We now show that the algorithm 
constructs an infinitely big polytope, 
solely with vertices generated by matrix products 
whose averaged spectral radius is equal to $\rho_c$.
Indeed, for $n\in\NN$, applying the sequence of products $A^n
$ to the starting vector $v_2$
we get the sequences of vector 
$A^n v_2=[\!\!\begin{array}{c c}n&1\end{array}\!\!]^T$,
where $\rho(A^n)^{1/n}=1$.
The same calculation shows that 
$B^n v_1=[\!\!\begin{array}{c c}1&n\end{array}\!\!]^T$
and $\rho(B^n)^{1/n}=1$.
Finally, $\comin V_{k}=
\comin \{
[\!\!\begin{array}{c c}1&k\end{array}\!\!]^T,
[\!\!\begin{array}{c c}k&1\end{array}\!\!]^T
\}
$, $k\in\NN$.
\end{example}

\subsection{Proof for Theorem~\ref{thm_termination}}\label{proof_termination}
\begin{proof}
    $(i)$ 
    Let $\delta=1$.
    Assume that the original invariant polytope algorithm~\ref{alg_invpoly} terminates at depth $N\in\NN$ with vertices $V_N^{orig}$,
    i.e.\ $\tilde{\mathcal{A}} \coast V_N^{orig} \subseteq \coast V_N^{orig}$.
    By construction of the original invariant polytope algorithm~\ref{alg_invpoly} and by~\eqref{equ_modinvpoly_selectnewvertex},
    there exists $K\in\NN$, $K\geq N$, such that
    $
    \coast V_N^{orig}=\coast \bigcup_{n=0}^N \tilde{\mathcal{A}}^n V_0 \subseteq \coast V_K^{mod}.
    $
    We claim that $\coast V_K^{mod}$ is an invariant polytope.
    By construction of the modified invariant polytope algorithm~\ref{alg_modinvpoly},
    $
    \coast V_K^{mod}\subseteq \coast \bigcup_{k=0}^K \tilde{\mathcal{A}}^k V_0.
    $
    By the invariance property of the polytope $\coast V_N^{orig}$ and by $K>N$,
    $
    \coast \bigcup_{k=0}^K \tilde{\mathcal{A}}^k V_0 = \coast V_N^{orig}.
    $
    It follows that $\coast V_N^{orig}=\coast V_K^{mod}$, and thus $\coast V_K^{mod}$ is an invariant polytope.
    
    The other direction follows similarly.
    
    $(ii)$ 
    Assume that $\JSR(\mathcal{A})<\delta^{-1}\cdot\rho_c$, or equivalently, $\JSR(\tilde{\mathcal{A}})<\gamma<1$ for some $\gamma>0$.
    \cite[Theorem I (b)]{BW92} implies that $\|\tilde{A}_{i_k}\cdots \tilde{A}_{i_1}\|\rightarrow 0$ for any product $\tilde{A}_{i_k}\cdots \tilde{A}_{i_1}\in\tilde{\mathcal{A}}^n$ as $k\rightarrow\infty$. 
    Thus, the modified invariant polytope algorithm eventually terminates.
    
    $(iii)$
    Let $k\in\NN_0$.
    Without loss of generality we assume that $1<b_{k+1}<b_{k}$.
    Let $v\in V_{k+1}$. We need to show that $\|\tilde{A}_j v\|_{\coast V_{k+1}}\leq b_{k+1}$ 
    for all $j\in\{1,\ldots,J\}$.
    If $\tilde{A}_j v\in V_{k+1}$, then we trivially get $\|\tilde{A}_j v\|_{\coast V_{k+1}}\leq1<b_{k+1}$.
    Thus, we assume that $\tilde{A}_j v\notin V_{k+1}$.
    Let $k'\in\NN_0$ be the iteration in which $N(\tilde{A}_j v)$ was computed.
    By~\eqref{equ_Wk}, $\coast W_{k'}\subseteq \coast V_{k+1}$.
    Therefore,
    $
    \|\tilde{A}_j v\|_{\coast V_{k+1}}\leq
    \|\tilde{A}_j v\|_{\coast V_{k'}}\leq
    \|\tilde{A}_j v\|_{\coast W_k'}=
    N(v)(1+\epsilon)^{-1}\leq b_{k+1}
    $.
\end{proof}

\section{Applications and numerical results}\label{sec_examples}
In this section we illustrate the modified Gripenberg algorithm~\ref{alg_modgrip} and the modified invariant polytope algorithm~\ref{alg_modinvpoly} with numerical examples.
For our tests we use matrices from standard applications, as well as random matrices.
We also try to repeat tests previously performed in the literature~\cite{BC11, BJP06,BC13,GP13,GP16,MOS01}.

The parameters for the various algorithms (ours and others) 
are chosen such that they terminate after a reasonably short time. 
For the modified invariant polytope algorithm~\ref{alg_modinvpoly}
the parameters are chosen such that the modified invariant polytope algorithm terminates at all, 
hopefully in shortest time.
We do not report the exact parameters, since we believe they are of no value for the reader.
The tests are performed using an
Intel Core i5-4670S@3.8GHz, 8GB RAM with the software
Matlab R2017a and
\href{https://www.gurobi.com}{Gurobi solver v8.0}.\footnote{
$(a)$ Our implementation also uses software containing functions from the \href{https://www.mathworks.com/matlabcentral/fileexchange/33202-the-jsr-toolbox}{JSR-Toolbox v1.2b}~\cite{Jung2014}. Permission to use has been kindly granted.
$(b)$ The \href{https://www.gurobi.com}{Gurobi solver} is free for academic use.
}

For the tests we report
$\bullet$ the dimension \emph{dim} of the matrices,
$\bullet$ the duration \emph{time} needed for the computation (this value is only to be understood in magnitudes),
$\bullet$ the number of matrices \emph{$J$} in the test set $\mathcal{A}$,
$\bullet$ the number of vertices \emph{\#V} of the invariant polytope,
$\bullet$ spectral maximizing product(s) \emph{s.m.p.}, and
$\bullet$ the number \emph{\#tests} of test runs.

\subsection{Main results}
\subsubsection{Modified invariant polytope algorithm}
To summarize, we can say that the single-threaded modified invariant polytope algorithm~\ref{alg_modinvpoly} 
is roughly three times faster than the original invariant polytope algorithm~\ref{alg_invpoly}.
If the dimension of the matrices is sufficiently large,
the parallelised modified invariant polytope algorithm~\ref{alg_modinvpoly}
scales nearly linearly with the number of available threads (for at least up to 16 threads). More precisely,
\begin{itemize}
\item for pairs of random matrices the modified invariant polytope algorithm~\ref{alg_modinvpoly} 
reports the exact value of the $\JSR$ in reasonable time up to dimension 25,
\item for Daubechies matrices the modified invariant polytope  algorithm 
reports the exact value of the $\JSR$ in reasonable time up to dimension 42,
\item for non-negative matrices it strongly depends on the problem. 
For random, sparse, non-negative matrices the modified invariant polytope algorithm 
works up to dimension $3000$ or higher. 
For the (sparse) matrices arising in the context of code capacities (Section~\ref{sec_capacity}) 
the modified invariant polytope algorithm works well only up to dimension $16$.
On the one hand this is due to the large number of matrices to be considered for these examples,
on the other hand the structure of the individual matrices seems to play a role.
\end{itemize}

\subsubsection{Modified Gripenberg algorithm}
For the modified Gripenberg algorithm~\ref{alg_modgrip} we can say, that it finds in almost all cases an s.m.p..
Thus, for fast estimates of the $\JSR$, the modified Gripenberg algorithm~\ref{alg_modgrip} may be used independently, 
e.g. in applications where the parameters where a matrix family has highest/lowest $\JSR$ need to be determined.
In a second step one then may compute the exact $\JSR$ for the found parameters using the modified invariant polytope algorithm.
 
Clearly, since the computation of the $\JSR$ is NP-hard, there must be sets of 
matrices for which the modified Gripenberg algorithm~\ref{alg_modgrip} fails\footnote{since the modified Gripenberg algorithm has polynomial complexity,} and we report mostly these cases together with a comparison with other algorithms.
These are

$\bullet$ the \emph{random Gripenberg} algorithm~\ref{alg_modgrip} described in Remark~\ref{rem_alg_randmodgrip},
$\bullet$ the \emph{Gripenberg} algorithm,
$\bullet$ the \emph{modified invariant polytope} algorithm~\ref{alg_modinvpoly} and
$\bullet$ the Monte-Carlo type \emph{genetic} algorithm~\cite{BC11}.

At least in our test runs, the modified Gripenberg algorithm~\ref{alg_modgrip} performs best, 
in the sense that in most cases it returns a correct s.m.p.\ in fastest time. More precisely,
for long s.m.p.s the modified Gripenberg algorithm~\ref{alg_modgrip} performs best and
for large sets of matrices the genetic algorithm and the modified invariant polytope algorithm~\ref{alg_modinvpoly} performs best.

\subsection{Randomly generated matrices}\label{sec_random}
We first present the behaviour of the modified invariant polytope algorithm~\ref{alg_modinvpoly} 
for pairs of matrices of dimensions $2$ to $20$ with normally distributed values whose 
$(a)$ matrices have the same 2-norm, 
$(b)$ matrices have the same spectral radius, and 
$(c)$ matrices have the same spectral radius and $\delta=0.99$
(where $\delta$ was the parameter controlling the accuracy of the modified invariant polytope algorithm~\ref{alg_modinvpoly},
 see Section~\ref{sec_modinvpoly_accuracy}~\eqref{alg_modinvpoly_accuracy}).
We see in Table~\ref{table_random_pairs} that the modified invariant polytope algorithm 
is applicable for pairs of random matrices up to dimension $25$, 
for which it takes roughly one weekend to complete.
For $\delta=0.95$ the modified invariant polytope algorithm is comparable to Gripenberg's algorithm.

Although the modified invariant polytope algorithm~\ref{alg_modinvpoly} produces 
polytopes with roughly twice as much vertices compared to the same test with the original invariant polytope algorithm in~\cite[Table~2]{GP13}, it still works very well for matrices of dimension 20.

\begin{table}[ht]
	\centering
	\caption[]{%
		Computation of the $\JSR$ for random pairs of matrices using the modified invariant polytope algorithm~\ref{alg_modinvpoly}.
		\emph{$\delta$}: accuracy parameter for the modified invariant polytope algorithm~\eqref{alg_modinvpoly_accuracy}, 
		\emph{dim}: dimension of the matrices, 
		\emph{\#V}: number of vertices of the invariant polytope, 
		\emph{time}: time needed to compute the invariant polytope, 
		\emph{$J$}: number of matrices, 
		\emph{$\#test$}: number of test runs.\newline
		\resizebox{\linewidth}{!}{$^{\dagger}$We print median values, since there are always outliers if $\delta\!=\!1$. The average values are roughly 100 times bigger.}
	}
	\begin{tabular}{|c|rr|rr|rr|} 
		\hline
		\multicolumn{7}{|c|}{\emph{$J=2$, $\#test=20$, median values$^{\dagger}$}}\\
		&\multicolumn{2}{|c|}{\emph{$(a)$ $\delta=1$}}&\multicolumn{2}{|c|}{\emph{$(b)$ $\delta=1$}}&\multicolumn{2}{|c|}{\emph{$(c)$ $\delta=0.99$}}\\ 
		&\multicolumn{2}{|c|}{\hphantom{imm}\emph{ equal norm\hphantom{mm}}}&\multicolumn{2}{|c|}{\emph{equal spectral radius}}&\multicolumn{2}{|c|}{\emph{equal spectral radius}}\\ 
		
		\emph{dim}&\emph{time}&\emph{\#V}&\emph{time}&\emph{\#V}&\emph{time}&\emph{\#V}\\ \hline  
		
		2    &  1.1$\,s$ &    5$\cdot 2$ &  1.2$\,s$ &    6$\cdot 2$ &  0.2$\,s$ &    5$\cdot 2$ \\       
		4    &  1.4$\,s$ &   17$\cdot 2$ &  1.8$\,s$ &   77$\cdot 2$ &  0.8$\,s$ &   19$\cdot 2$ \\       
		6    &  2.0$\,s$ &   47$\cdot 2$ &  2.5$\,s$ &  130$\cdot 2$ &  1.5$\,s$ &   47$\cdot 2$ \\       
		8    &  2.5$\,s$ &  100$\cdot 2$ &  3.9$\,s$ &  220$\cdot 2$ &  2.1$\,s$ &   98$\cdot 2$ \\       
		10    &  4.9$\,s$ &  270$\cdot 2$ &  5.1$\,s$ &  320$\cdot 2$ &  3.3$\,s$ &  220$\cdot 2$ \\       
		12    &  4.7$\,s$ &  280$\cdot 2$ &   11$\,s$ &  770$\cdot 2$ &  6.6$\,s$ &  570$\cdot 2$ \\       
		14    &  8.4$\,s$ &  510$\cdot 2$ &   21$\,s$ & 1100$\cdot 2$ &   12$\,s$ &  800$\cdot 2$ \\       
		16    &   25$\,s$ & 1100$\cdot 2$ &   33$\,s$ & 1400$\cdot 2$ &   25$\,s$ & 1000$\cdot 2$ \\       
		18    &   90$\,s$ & 2100$\cdot 2$ &  200$\,s$ & 2500$\cdot 2$ &   44$\,s$ & 1600$\cdot 2$ \\       
		20    &  295$\,s$ & 3100$\cdot 2$ & 5000$\,s$ & 6200$\cdot 2$ &  800$\,s$ & 3900$\cdot 2$ \\ \hline
	\end{tabular}
	\label{table_random_pairs}
\end{table}

Random matrices with non-negative entries are a worthy test case, since the computation of the invariant polytope 
(i.e.\ the main loop in the modified invariant polytope algorithm~\ref{alg_modinvpoly}) 
always finishes after a few seconds, nearly regardless of the dimension. 
Since the implementation is not optimized for such high dimensions, 
the modified invariant polytope algorithm still needs some minutes to terminate, mostly due to the preprocessing steps~\eqref{alg_modinvpoly_modgrip}-\eqref{alg_modinvpoly_balancing}.
For sparse matrices with non-negative entries, 
the modified invariant polytope algorithm~\ref{alg_modinvpoly} performs slightly worse, 
but is still applicable up to dimension $2000$ or higher. Again, it is very likely that it still works for even larger matrices if the implementation were optimized for such matrices,
see Table~\ref{table_random_pairs_nonneg} for the results. We again give the median values. The average values for these cases are roughly 10\% higher.
Another benchmark for non-negative matrices is presented in Section~\ref{sec_capacity}.

\begin{table}[ht]
	\centering
	\caption{Computation of the $\JSR$ using the modified invariant polytope algorithm~\ref{alg_modinvpoly}
		for random pairs of matrices with non-negative entries.
		\emph{dim}: dimension of the matrices, 
		\emph{$J$}: number of matrices, 
		\emph{$\#test$}: number of test runs.
		\emph{time}: time needed to compute the invariant polytope, 
		\emph{\#V}: number of vertices of the 
		polytope.
		\newline
		$^{\dagger}$We print the median values, since there are always outliers if $\delta=1$. 
		The average values are roughly 100 times bigger.    
		$^{\dagger\dagger}$Since the matrices are random, most of the sparse matrices have non-trivial invariant subspaces which reduces the effective dimension of the matrices by roughly 10\%.
		$^{\dagger\dagger\dagger}$Most cones have 8 or 16 vertices, because the algorithm terminates after 3 or 4 iterations. The algorithm does not check whether all of these vertices are really outside of the polytope.
	}    
	\begin{tabular}{|c|rr|rr|rr|rr|}
		\hline
		\multicolumn{9}{|c|}{\emph{$J=2$, $\#test=20$, non-negative entries, equal spectral radius, 
				median values$^{\dagger}$}}                \\ 
		&\multicolumn{2}{|c|}{\emph{ 0\%~sparsity}}&\multicolumn{2}{|c|}{\emph{ 90\%~sparsity}}&\multicolumn{2}{|c|}{\emph{ 98\%~sparsity}}&\multicolumn{2}{c|}{\emph{ 99\%~sparsity}}\\
		\emph{dim}$^{\dagger\dagger}$&\emph{\hphantom{m}time}&\emph{\hphantom{m}\#V$^{\dagger\dagger\dagger}$}&\emph{\hphantom{m}time}&\emph{\hphantom{m}\#V$^{\dagger\dagger\dagger}$}&\emph{\hphantom{m}time}&\emph{\hphantom{mm}\#V}&\emph{\hphantom{m}time}&\emph{\hphantom{mm}\#V}\\ \hline  
		20    &      0.3$\,s$ &     7        &      1.7$\,s$ &        42    &               &              &               &        \\        
		50    &      0.3$\,s$ &     8        &      1.6$\,s$ &        50    &      2.2$\,s$ &          50  &               &        \\        
		100   &      0.4$\,s$ &     8        &      0.8$\,s$ &        25    &       17$\,s$ &        1300  &               &        \\        
		200   &      0.5$\,s$ &     8        &      1.0$\,s$ &        23    &      5.0$\,s$ &         220  &      110$\,s$ &   2600 \\        
		500   &      1.2$\,s$ &     8        &      1.8$\,s$ &        16    &      7.7$\,s$ &          90  &       26$\,s$ &    310 \\        
		1000  &      6.3$\,s$ &     8        &       11$\,s$ &        16    &       30$\,s$ &          45  &       72$\,s$ &    110 \\        
		2000  &       35$\,s$ &     8        &       72$\,s$ &        16    &       35$\,s$ &           8  &      290$\,s$ &    64  \\        \hline
	\end{tabular}
	\label{table_random_pairs_nonneg}  
\end{table}

In Table~\ref{table_random_modgrip} we see how the modified Gripenberg algorithm~\ref{alg_modgrip} 
performs on random matrices with equally distributed values in $[-5,\ 5]$ to mimic the test in~\cite[Section 4.2]{BC11}.
Interestingly, the genetic algorithm performs very bad, as does the \emph{random} modified Gripenberg algorithm.
We report the \emph{succes}-rate, i.e.\ how often the algorithms did find an s.m.p.\ in percent.

\begin{table}[ht]
	\centering
	\caption{Performance of various algorithms searching for s.m.p.s.
		We use the modified invariant polytope algorithm to test, whether the found s.m.p.-candidates are indeed s.m.p.s.
		\emph{dim}: dimension of the matrices, 
		\emph{$J$}: number of matrices, 
		\emph{success}: percentage of how often a correct s.m.p. is found.
		\emph{$\#test$}: number of test runs.
		\emph{time}: time needed by the algorithm.
	}
	\begin{tabular}{|l|rr|rr|rr|}
		\hline
		\multicolumn{7}{|c|}{$\# tests=100$}   \\
		&\multicolumn{2}{|c|}{$J=2$, $dim=2$}&\multicolumn{2}{|c|}{$J=4$, $dim=4$}&\multicolumn{2}{|c|}{$J=8$, $dim=8$}   \\
		Algorithm               & \emph{success} & \emph{time}& \emph{success} & \emph{time}& \emph{success} & \emph{time}\\
		\hline
		mod. invariant polytope & 100\%	&	$1.1\,s$ & 100\% & $4.3\,s$   & 100\% & $40.0\,s$\\
		mod. Gripenberg         & 100\%	&	$1.9\,s$ & 100\% & $4.1\,s$   & 100\% &  $5.4\,s$\\
		random Gripenberg       & 100\%	&	$1.8\,s$ &  99\% & $3.8\,s$   &  82\% &  $4.3\,s$\\
		Gripenberg              & 100\% &	$3.8\,s$ & 100\% & $20.3\,s$  & 100\% & $82.1\,s$\\
		brute force             & 100\% &	$180\,s$ &  98\% & $180.0\,s$ & 74\% & $180.0\,s$\\
		genetic                 & 100\% &	$7.1\,s$ &  97\% & $9.3\,s$   & 87\% &  $12.0\,s$\\ \hline
	\end{tabular}
	\label{table_random_modgrip}
\end{table}


\subsection{Handpicked generic matrices}
\begin{example}\label{ex_119}
Let 
{\small
\begin{equation*}
X_1=
\left[\!\begin{array}{r r}
\dfrac{15}{92} & \dfrac{-73}{79}\\[2.ex]
\dfrac{56}{59}  & \dfrac{89}{118}\end{array}\!\right]
,\
X_2=
\left[\!\begin{array}{c c}
\dfrac{-231}{241} & \dfrac{-143}{219}\\[2.ex]
\dfrac{103}{153}  & \dfrac{-38}{65}\end{array}\!\right]
.
\end{equation*}
}
The set $\mathcal{X}=\{X_1,X_2\}$ has an s.m.p.\ of length 119 with normalized spectral radius
 $\JSR(\mathcal{\mathcal{X}})\simeq 1.01179$. 
Gripenberg's algorithm finds an s.m.p.\ after an evaluation of $\sim$630k products, taking roughly ten minutes. 
Both the modified Gripenberg algorithm, as well as the genetic algorithm fail.
The modified invariant polytope algorithm~\ref{alg_modinvpoly} finds an s.m.p.\ after less than one minute.
The test results are in Table~\ref{table_119}.
\end{example}

\begin{table}[tb]
\centering
\caption{Performance of various algorithms searching for s.m.p.s for a particular hard problem.
For the test set $\mathcal{X}$ (Example~\ref{ex_119}) all fast algorithms fail.
\emph{dim}: dimension of the matrices, 
\emph{lower bd.}: computed lower bound for the $\JSR$,
\emph{$J$}: number of matrices, 
\emph{time}: time needed by the algorithm.
}
\begin{tabular}{|l|l|c|r|}
\hline
\emph{Testset} & \emph{Algorithm}  & \emph{lower bd.} & \emph{time} \\
\hline
$\mathcal{X}$           & mod. invariant polytope & $1.01179\ldots$     & $40\,s$ \\
$J=2$                   & mod. Gripenberg         & $1.011\underline{3}0\ldots$     & $4\,s$   \\
$dim=2$                 & random Gripenberg       & $1.0117\underline{2}\ldots$     & $10\,s$   \\
                        & Gripenberg              & $1.01179\ldots$     & $580\,s$\\
                        & genetic                 & $1.011\underline{3}0\ldots$     & $8\,s$   \\
\hline
\end{tabular}
\label{table_119}
\end{table}

Example~\ref{ex_longsmp} is of interest because it is a rather simple family of two matrices with an arbitrary long s.m.p..
\begin{example}\label{ex_longsmp}
Let $n\in\NN$,
$C_0=\left[\begin{array}{c c}1 & 1\\0&1\end{array}\right]$ and
${C}_n=\left[\begin{array}{c c}0 & 0\\\tfrac{1}{n}e^{1+\frac{1}{n}}& 0\end{array}\right]$,
Then $C_0^{n}C_n$ is an s.m.p.\ for the set $\mathcal{C}_n=\{C_0,C_n\}$ with $\JSR(\mathcal{C}_n)=e^{1/n}$.

The genetic algorithm fails for most matrices of that family. All other algorithms report the correct s.m.p.\ in less than $5\,s$.
The test results are in Table~\ref{table_longsmp}.
\end{example}
\begin{proof}[Proof for Example~\ref{ex_longsmp}]\label{proof_longsmp}
Define $\tilde{C}_n=\left[\begin{array}{c c}0 & 0\\n& 0\end{array}\right]$, $n\in\NN$.
A product of $C_0$ and $\tilde{C}_n$ is non-zero if and only if it is of the form 
$
C_0^{i_1}\tilde{C}_n C_0^{i_2}\tilde{C}_n \cdots  \tilde{C}_n C_0^{i_m}. 
$
Since the spectral radius does not change under cyclic permutation, we can assume that the product is of the form
$C_0^{i_1}\tilde{C}_n C_0^{i_2}\tilde{C}_n \cdots C_0^{i_m}\tilde{C}_n$. A (lengthy) straightforward computation shows that the normalized spectral radius of this product is 
$(n^m \prod_{j=1}^m i_j)^{1/(m+\sum_{j=1}^m i_j)}$. Taking the gradient with respect to $i$ and setting it to zero, we immediately get that all $i_j$ must be equal. 
Thus, the normalized spectral radius of all finite products is maximized with a product of the form $C_0^m \tilde{C}_n$ whose normalized spectral radius equals $mn^{1/(1+m)}$. For fixed $m\in\NN$ this term has its maximum at 
$n=\frac{1}{m} e^{1+1/m}$.
Thus, $C_0^nC_n$ is the product with largest normalized spectral radius under all finite products.
Using~\eqref{equ_generalized_jsr} we conclude that $\JSR(\mathcal{C})=\rho(C_0^nC_n)^{1/(n+1)}=(e^{(n+1)/n})^{1/(n+1)}=e^{1/n}$.
\end{proof}

\begin{table}[tb]
\centering
\caption{Performance of various algorithms searching for s.m.p.s.
	For the test sets $\mathcal{C}_{n}$ (Example~\ref{ex_longsmp}) the genetic algorithm mostly fails.
    \emph{dim}: dimension of the matrices, 
    \emph{lower bd.}: computed lower bound for the $\JSR$,
    \emph{$J$}: number of matrices, 
    \emph{s.m.p.}: an s.m.p.,
    \emph{time}: time needed to compute the invariant polytope, 
}
\begin{tabular}{|l|l|c|r|}
\hline
\emph{Test set} & \emph{Algorithm}  & \emph{lower bd.} & \emph{time} \\
\hline
$\mathcal{C}_{15}$      & mod. invariant polytope & $1.0689\ldots$     & $1.7\,s$ \\
$J=2$                   & mod. Gripenberg         & $1.0689\ldots$     & $3.3\,s$ \\
$dim=2$                 & random Gripenberg       & $1.0689\ldots$     & $3.2\,s$ \\
$s.m.p.=C_0^{15}C_{15}$ & Gripenberg              & $1.0689\ldots$     & $0.1\,s$ \\
                        & genetic                 & $1.0689\ldots$     & $7.0\,s$ \\
\hline
$\mathcal{C}_{30}$      & mod. invariant polytope & $1.0338\ldots$     & $2.5\,s$ \\
$J=2$                   & mod. Gripenberg         & $1.0338\ldots$     & $4.0\,s$ \\
$dim=2$                 & random Gripenberg       & $1.0338\ldots$     & $4.3\,s$ \\
$s.m.p.=C_0^{30}C_{30}$ & Gripenberg              & $1.0338\ldots$     & $0.1\,s$ \\
                        & genetic                 & $1.0\underline{2}15\ldots$     & $6.6\,s$ \\
\hline
$\mathcal{C}_{60}$      & mod. invariant polytope & $1.0168\ldots$     & $4.0\,s$ \\
$J=2$                   & mod. Gripenberg         & $1.0168\ldots$     & $3.1\,s$ \\
$dim=2$                 & random Gripenberg       & $1.0168\ldots$     & $4.3\,s$ \\
$s.m.p.=C_0^{60}C_{60}$ & Gripenberg              & $1.0168\ldots$     & $0.1\,s$ \\
                        & genetic                 & $1.0\underline{0}00\ldots$     & $6.3\,s$ \\
\hline
\end{tabular}
\label{table_longsmp}
\end{table}

\subsection{Capacity of codes with forbidden difference sets}\label{sec_capacity}

In some electromagnetic recording systems, the bit error rate is often dominated by a small set of certain \emph{forbidden difference patterns} $D$. 
Thus, one needs to construct sets of allowed words with values in $\{0,1\}$, all of whose possible differences do not yield such a forbidden pattern. Clearly, one wants codes which constrain the number of all possible patterns as least as possible. We are interested in how constraining a given forbidden difference pattern is, which we denote as the \emph{capacity} $\operatorname{cap} D\in[0,\ 1]$. 
The larger the capacity, the better. This problem can be expressed in terms of the $\JSR$ of a finite set of matrices. See~\cite{MOS01} for more details.
The occurring matrices in this application only have entries in $\{0,1\}$, but their dimension, as well as the number of matrices increases exponentially with the length of the forbidden difference patterns, e.g.\ for
\newcommand{\tone}[0]{\!\text{ \scalebox{0.9}{\raisebox{0.25 ex}{1}}}} 
\newcommand{\tzer}[0]{\!\text{ \scalebox{0.9}{\raisebox{0.25 ex}{0}}}} 
    $D=\{\nn\pp\mm\}$ the capacity of $D$ is given by
    \begin{align*}
    \operatorname{cap} D=\log_2\JSR
    \Big(
    \Bigg\{
    \left[\!\begin{array}{cccc}    
    \tone&\tzer&\tone&\tzer\\
    \tzer&\tzer&\tzer&\tzer\\
    \tzer&\tone&\tzer&\tone\\
    \tzer&\tone&\tzer&\tone
    \end{array}\!\right],
    \left[\!\begin{array}{cccc}    
    \tone&\tzer&\tone&\tzer\\
    \tzer&\tzer&\tone&\tzer\\
    \tzer&\tone&\tzer&\tzer\\
    \tzer&\tone&\tzer&\tone
    \end{array}\!\right],
    \left[\!\begin{array}{cccc}    
    \tone&\tzer&\tone&\tzer\\
    \tone&\tzer&\tzer&\tzer\\
    \tzer&\tzer&\tzer&\tone\\
    \tzer&\tone&\tzer&\tone
    \end{array}\!\right],
    \left[\!\begin{array}{cccc}    
    \tone&\tzer&\tone&\tzer\\
    \tone&\tzer&\tone&\tzer\\
    \tzer&\tzer&\tzer&\tzer\\
    \tzer&\tone&\tzer&\tone
    \end{array}\!\right]
    \Bigg\}
    \Big).
    \end{align*}

We use the modified invariant polytope algorithm to compute the capacities for the forbidden difference patterns $D$ taken 
from~\cite[p. 10]{MOS01}, \cite[Table~1]{BC11}, \cite[p. 6]{BJP06} and for difference sets with the additional symbol $\ppmm$, denoting $+1$ and $-1$, discussed in~\cite[Section v]{BJP06}. Nearly all of these capacities were not known exactly before.

For most difference sets $D$, there are several s.m.p.s., that not only share the same leading eigenvalue but also the same eigenvector. Due to this reason, the modified invariant polytope algorithm~\ref{alg_modinvpoly} sometimes only gives a bound for the $\JSR$ up to the accuracy in which we can compute the norms $\|\tilde{A_j} v\|_{\coast W}$.
We implemented the Matlab routine \texttt{codecapacity} which computes the set of matrices needed for the $\JSR$ computation for a given difference set $D$. It works for reasonably small difference sets, and theoretically also for difference words with entries in $\{-K,\ldots,K\}$, $K\in\NN$.

The exact computation of the capacity using the modified invariant polytope algorithm~\ref{alg_modinvpoly} was only possible if we used the estimates for the Minkowski norm in Lemma~\ref{thm_estimate_1_pos}~\eqref{thm_estimate_1}, which reduced the norms to be computed by a factor of 100.


The difference set $D_4=\{\Circ\Circ\!+\Circ-\}$, taken from~\cite[Table~1]{BC11}, 
is a good test case for the modified Gripenberg algorithm, 
since the computation of the capacity translates to the $\JSR$ of a set with $256$ matrices of dimension $16$. 
As one can expect, Gripenberg's algorithm fails to find an s.m.p., 
also the modified Gripenberg algorithm~\ref{alg_modgrip} fails.
The genetic algorithm in most cases finds a better product than the one found by Gripenberg's algorithm. 
The modified invariant polytope algorithm~\ref{alg_modinvpoly} also finds that better product after a while, 
but it did not terminate in reasonable time. Thus, the exact capacity,
and whether an s.m.p.\ exists is still unknown. 
The test results are in Table~\ref{table_capacity} and~\ref{table_capacity_modgrip}.

\begin{table}[tb]
\centering
\caption{Capacity of various difference sets $D$.
 $\epsilon=10^{-10}$: computational accuracy,
\emph{$\operatorname{cap}(D)$}: capacity,
\emph{D}: set of forbidden differences,
\emph{$dim$}: dimension of the matrices, 
\emph{$J$}: number of matrices, 
\emph{\#V}: number of vertices of the invariant polytope, 
\emph{s.m.p.}: an s.m.p..
\newline
$^?$For some sets $D$ the modified invariant polytope algorithm~\ref{alg_modinvpoly} did not terminate, thus the given product is not proven to be an s.m.p.
}
\begin{tabular}{|lclrrr|}
\hline
\emph{$D$}      &      \emph{s.m.p.}  & \emph{$\operatorname{cap}(D)$} &     \emph{\#V} & \emph{$J$} & \emph{$dim$}   \\
\hline 
$\ppmm\ppmm$          &         $B_2B_3$         & $1/2$                        &     $3\cdot 2$ &    4 &   2   \\
$\nn\ppmm$            &          $B_3$           & $0$                          &     $2\cdot 2$ &    4 &   2   \\
\hline   
$\nn\pp\mm$           &         $B_4B_1$         & $0.6942\ldots$               &    $45\cdot 2$ &    4 &   4   \\
$\nn\pp\pp$           &           $B_2{}^?$            & $0.6942\ldots+[0,\epsilon]$  &    $25\cdot 2$ &    4 &   4   \\
$\nn\ppmm\ppmm$       &         $B_1B_2$         & $1/2$                        &    $37\cdot 2$ &   16 &   4   \\
$\ppmm\ppmm\ppmm$     &       $B_6B_4B_1$        & $2/3$                        &    $19\cdot 2$ &   16 &   4   \\ 
\hline
$\pp\mm\pp\mm$        &         $B_1B_2$         & $0.9468\ldots$               &    $86\cdot 2$ &    2 &   8   \\
$\pp\pp\pp\mm$        &         $B_1B_2$         & $0.9005\ldots$               &    $40\cdot 2$ &    2 &   8   \\
$\pp\pp\pp\pp$        &          $B_1$           & $0.9468\ldots$               &    $84\cdot 2$ &    2 &   8   \\
$\nn\pp\mm\pp$        &          $B_3$           & $0.8791\ldots$               &    $43\cdot 2$ &    4 &   8   \\
$\nn\pp\pp\mm$        &          $B_3$           & $0.8113\ldots$               &    $46\cdot 2$ &    4 &   8   \\
$\nn\pp\pp\pp$        &          $B_1$           & $0.8791\ldots$               &    $46\cdot 2$ &    4 &   8   \\
$\nn\pp\pp\ppmm$      &       $B_1^2B_2^2$       & $0.7396\ldots$               &   $244\cdot 2$ &   16 &   8   \\
$\nn\pp\nn\pp$        & $B_4B_{11}^2B_{13}B_6^2$ & $0.7298\ldots$               &   $804\cdot 2$ &   16 &   8   \\
$\nn\pp\nn\ppmm$      &   $B_{16}B_{52}B_{103}{}^?$ & $2/3+[0,\epsilon]$                        & $23152\cdot 2$ &  256 &   8   \\
$\ppmm\ppmm\ppmm\ppmm$& $B_{86}B_{52}B_{16}B_1$  & $3/4$                        &   $357\cdot 2$ &  256 &   8   \\
\hline  
$\nn\pp\mm\pp\nn$     &      $B_{11}B_{13}$      & $0.9163\ldots$               &  $1721\cdot 2$ &   16 &  16   \\
$\nn\pp\pp\pp\nn$     &         $B_4B_6$         & $0.9163\ldots$               &  $4559\cdot 2$ &   16 &  16   \\
\hline
$\nn\pp\pp\pp\pp\nn$  &            $B_2{}^?$         & $0.9614\ldots+[0,\epsilon]$               & $17902\cdot 2$ &   16 &  32   \\
\hline
$\pp\pp\pp\pp\pp\mm\nn$ & 			$B_3$        & $0.9761\ldots$ 				& $992\cdot 2$   &    4 &  64 \\
\hline
\end{tabular}
\label{table_capacity}
\end{table}

\begin{table}[tb]
\centering
\caption{Performance of various algorithms searching for s.m.p.s for a particular hard problem.
	The modified Gripenberg algorithm~\ref{alg_modgrip} fails for the set of matrices corresponding to the forbidden difference set $D_4$ in Section~\ref{sec_capacity}.
\emph{$dim$}: dimension of the matrices, 
\emph{lower bd.}: computed lower bound for the $\JSR$,
\emph{$J$}: number of matrices, 
\emph{time}: time needed by the algorithm.
}
\begin{tabular}{|l|l|c|r|}
\hline
\emph{Testset} & \emph{Algorithm}  & \emph{lower bd.} & \emph{time} \\
\hline
$D_4=\{\Circ\Circ\!+\Circ-\}$   & mod. invariant polytope & $1.6736\ldots$     & $40\,s$ \\
$J=256$                         & mod. Gripenberg         & $1.6\underline{6}63\ldots$     & $2\,s$   \\
$dim=16$                        & random Gripenberg       & $1.6\underline{6}63\ldots$     & $2\,s$   \\
                                & Gripenberg              & $1.6\underline{6}63\ldots$     & $60\,s$\\
                                & genetic                 & $1.6736\ldots$     & $10\,s$   \\    
\hline
\end{tabular}
\label{table_capacity_modgrip}
\end{table}

\subsection{H\"older exponents of Daubechies wavelets}\label{sec_daub}
An important application of the $\JSR$ is the computation of the regularity of \emph{refinable functions}.
These are functions $\phi\in C_0(\RR^s)$ which fulfil a functional equation of the form
 $\phi(x)=
\sum_{\alpha\in\ZZ^s}\allowbreak
a(\alpha)
\phi(2x-\alpha)$, 
$x\in\RR$, with $a\in\ell_0(\ZZ^s)$. 
We use the modified invariant polytope algorithm~\ref{alg_modinvpoly} to compute the Hölder regularity of the Daubechies wavelets $D_n$~\cite{D88}. The regularity of $D_2, D_3$, and $D_4$  was 
computed by Daubechies and Lagarias~\cite{DL1992},
Gripenberg~\cite{Grip96} computed it for $D_5, \ldots , D_8$, 
then Guglielmi and Protasov~\cite{GP16}, as a demonstration of the original invariant polytope algorithm,
computed the regularity of $D_9, \ldots , D_{20}$. 
Now with the modified invariant polytope algorithm,  
we can compute the H\"older regularity for Daubechies wavelets up to $D_{42}$.

As noted in~\cite[Section 6.2]{GP16}, the polytopes generated by these matrices are very flat 
and the introduction of nearly-candidates and extra-vertices tremendously increases the performance of the invariant polytope algorithm.
Respectively, using the wrong set of nearly-candidates, 
the modified invariant polytope algorithm did not terminate at all. 
These cases are marked with $\dagger$ in Table~\ref{table_daub_matrices}. 
The right nearly-candidates and extra-vertices,
i.e.\ good values for $\tau$, $T$, $B_{\text{nearly}}$ and $B_{\text{extra}}$,
were merely found by trial and error.
We report the number of extra-vertices and the vertices of the roots from the nearly-s.m.p.s together under \emph{\#Extra-V}. The number of the invariant polytopes vertice's is depicted in Figure~\ref{fig_daub} (left side).

\begin{remark}
With the new values for $D_{21}$ to $D_{42}$ we can refine the observation in~\cite{GP15poin}, 
that the differences of H\"older regularities $\alpha_n-\alpha_{n-1}$ seem to converge towards a value of $0.21$ or maybe even $0.2$, 
see Figure~\ref{fig_daub} (right side).
\end{remark}


\begin{figure}[t]
	\centering
	\includegraphics{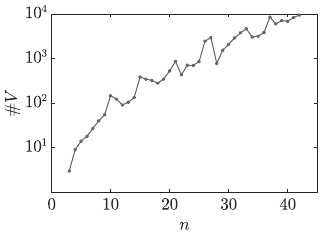}$\qquad$
   	\includegraphics{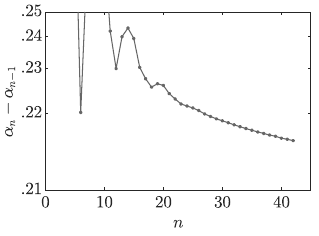}
    \vspace*{-.5em}
	\caption{Left: Number of vertices of the polytope $\#V\!$ against index of Daubechies wave\-let $D_n$. 
    Right: Difference of regularities $\alpha$ of consecutive Dau\-bechies wave\-lets.}
    \Description[The increase in vertices is non-monotone and seems to be exponential.]%
    {The increase in vertices is non-monotone and seems to be exponential.}
    \vspace{-1.5em}
	\label{fig_daub}
\end{figure}

\begin{table}[p]
\centering
\caption{H\"older regularity of Daubechies wavelets.
   \emph{$\alpha$}: H\"older regularity of Daubechies wavelet,
    \emph{$D_n$}: index of Daubechies wavelet,
    \emph{\#V}: number of vertices of the invariant polytope, 
    \emph{\#Extra-V}: number of extra-vertices including those from nearly-s.m.p.s.,
    \emph{s.m.p.}: an s.m.p.,
    \emph{time}: time needed to compute the invariant polytope.
For the cases marked with $\dagger$, using the wrong set of nearly-candidates, the algorithm did not terminate at all.
}
\begin{tabular}{|c|ccrrc|}
	\hline
\emph{~$D_n$~}   &  ~~~~\emph{s.m.p.}~~~~  & \emph{\#Extra-V} &    \emph{\#V}  &    \emph{time} &  \emph{$\alpha$} \\ \hline
            2  &      $B_0$      &      0           &     0$\cdot 2$ &        $<5\,s$ & $~0.55001\ldots$ \\ 
            3  &      $B_0$      &      0           &     3$\cdot 2$ &        $<5\,s$ & $~1.08783\ldots$ \\	
            4  &      $B_0$      &      2           &     9$\cdot 2$ &        $<5\,s$ & $~1.61793\ldots$ \\	
            5  & $B_0$ and $B_1$ &      2           &    14$\cdot 2$ &        $<5\,s$ & $~1.96896\ldots$ \\	
            6  & $B_0$ and $B_1$ &      3           &    18$\cdot 2$ &        $<5\,s$ & $~2.18914\ldots$ \\	
            7  & $B_0$ and $B_1$ &      4           &    27$\cdot 2$ &        $<5\,s$ & $~2.46041\ldots$ \\	
            8  & $B_0$ and $B_1$ &      5           &    40$\cdot 2$ &        $<5\,s$ & $~2.76082\ldots$ \\	
            9  & $B_0$ and $B_1$ &      6           &    55$\cdot 2$ &        $<5\,s$ & $~3.07361\ldots$ \\	
            10 &  $B_0^2B_1^2$   &      5           &   147$\cdot 2$ &        $<5\,s$ & $~3.36139\ldots$ \\	
            11 & $B_0$ and $B_1$ &      8           &   123$\cdot 2$ &         7$\,s$ & $~3.60347\ldots$ \\	
            12 & $B_0$ and $B_1$ &      9           &    91$\cdot 2$ &         7$\,s$ & $~3.83348\ldots$ \\	
            13 & $B_0$ and $B_1$ &     10           &   105$\cdot 2$ &         6$\,s$ & $~4.07348\ldots$ \\	
            14 & $B_0$ and $B_1$ &     11           &   134$\cdot 2$ &         8$\,s$ & $~4.31676\ldots$ \\	
            15 &  $B_0^4B_1^2$   &     11           &   386$\cdot 2$ &         6$\,s$ & $~4.55612\ldots$ \\	
            16 &  $B_0^2B_1^2$   &     12           &   346$\cdot 2$ &         7$\,s$ & $~4.78644\ldots$ \\	
            17 & $B_0$ and $B_1$ &     14           &   324$\cdot 2$ &         5$\,s$ & $~5.01380\ldots$ \\	
            18 & $B_0$ and $B_1$ &     15           &   282$\cdot 2$ &         8$\,s$ & $~5.23917\ldots$ \\	
            19 & $B_0$ and $B_1$ &     16           &   346$\cdot 2$ &         9$\,s$ & $~5.46532\ldots$ \\	
            20 & $B_0$ and $B_1$ &     17           &   529$\cdot 2$ &        12$\,s$ & $~5.69108\ldots$ \\	
            21 &  $B_0^2B_1^2$   &     17           &   868$\cdot 2$ &        15$\,s$ & $~5.91500\ldots$ \\	
22${^\dagger}$ &  $B_0^2B_1^4$   &     22           &   433$\cdot 2$ &         9$\,s$ & $~6.13779\ldots$ \\	
            23 & $B_0$ and $B_1$ &     20           &   707$\cdot 2$ &        18$\,s$ & $~6.35958\ldots$ \\	
            24 & $B_0$ and $B_1$ &     21           &   701$\cdot 2$ &        16$\,s$ & $~6.58096\ldots$ \\	
            25 & $B_0$ and $B_1$ &     22           &   861$\cdot 2$ &        20$\,s$ & $~6.80198\ldots$ \\	
            26 &  $B_0^4B_1^2$   &     22           &  2471$\cdot 2$ &        73$\,s$ & $~7.02250\ldots$ \\	
            27 &  $B_0^2B_1^2$   &     29           &  2952$\cdot 2$ &        60$\,s$ & $~7.24241\ldots$ \\	
28$^{\dagger}$ &  $B_0^2B_1^6$   &    105           &   777$\cdot 2$ &        24$\,s$ & $~7.46187\ldots$ \\	
            29 & $B_0$ and $B_1$ &     26           &  1545$\cdot 2$ &        39$\,s$ & $~7.68091\ldots$ \\	
            30 & $B_0$ and $B_1$ &     27           &  2078$\cdot 2$ &        64$\,s$ & $~7.89962\ldots$ \\	
            31 & $B_0$ and $B_1$ &     29           &  2898$\cdot 2$ &       190$\,s$ & $~8.11801\ldots$ \\	
            32 &  $B_0^2B_1^2$   &     29           &  3791$\cdot 2$ &       760$\,s$ & $~8.33605\ldots$ \\	
33$^{\dagger}$ &  $B_0^2B_1^2$   &     30           &  4692$\cdot 2$ &      1330$\,s$ & $~8.55379\ldots$ \\	
            34 & $B_0$ and $B_1$ &     32           &  3047$\cdot 2$ &       628$\,s$ & $~8.77123\ldots$ \\	
            35 & $B_0$ and $B_1$ &     33           &  3191$\cdot 2$ &       727$\,s$ & $~8.98841\ldots$ \\	
            36 & $B_0$ and $B_1$ &     34           &  3887$\cdot 2$ &       881$\,s$ & $~9.20533\ldots$ \\	
            37 &   $B_0^6B_1^2$  &     70           &  8529$\cdot 2$ &      6503$\,s$ & $~9.42202\ldots$ \\	
            38 &   $B_0^2B_1^2$  &     38           &  6035$\cdot 2$ &      3540$\,s$ & $~9.63847\ldots$ \\	
            39 &   $B_0^2B_1^4$  &     40           &  7142$\cdot 2$ &      3900$\,s$ & $~9.85474\ldots$ \\	
            40 & $B_0$ and $B_1$ &     38           &  6909$\cdot 2$ &      5550$\,s$ & $10.07073\ldots$ \\ 
            41 & $B_0$ and $B_1$ &     39           &  8343$\cdot 2$ &      8743$\,s$ & $10.28656\ldots$ \\ 
            42 & $B_0$ and $B_1$ &     40           &  9508$\cdot 2$ &     16373$\,s$ & $10.50220\ldots$ \\ 
        \hline
\end{tabular}
\label{table_daub_matrices}
\end{table}

\section{Conclusion and further work}
\subsection{Conclusion}
The modified Gripenberg algorithm~\ref{alg_modgrip} 
together with the modified invariant polytope algorithm~\ref{alg_modinvpoly} 
can compute the exact value of the $\JSR$ in a short time (less than 30~minutes) 
for most matrix families up to dimension $22$, in some cases even up to dimension $40$. 
For matrices with non-negative entries, the modified invariant polytope algorithm may work up to a dimension of $3000$.
Even more, since the modified Gripenberg algorithm~\ref{alg_modgrip} finds in almost all cases a correct s.m.p., 
it may be used alone for fast estimates of the $\JSR$ in time critical applications.

\subsection{Further work}
From the mathematical point of view, the question why the modified Gripenberg algorithm~\ref{alg_modgrip} works so well is of interest, in particular why it works mostly better than the \emph{random} Gripenberg algorithm.
It also may be useful to search for better estimates for the Minkowski norms, e.g.\ with orthant-monotonic norms, which would lead to a considerable speed up of the modified invariant polytope algorithm.
%

From the algorithmic point of view, the modified invariant polytope algorithm could be made faster by using approximate solutions to the LP-problem when computing the Minkowski-norms, since the exact value of the norms is of minor interest ---
for the modified invariant polytope algorithm it is enough to know whether a point is inside or outside of the polytope.


We plan to implement the case $(C)$ of complex leading eigenvalue in the near future 
and optimize the modified invariant polytope algorithm for a large number of parallel threads.
Case $(C)$ occurs seldom, in the sense that we did not encounter a set of matrices of practical interest 
with complex leading eigenvectors yet.


\begin{acks}
The author is grateful for the hospitality, help and encouragement of Prof.~V.~Yu.~Protasov.
The work is supported by the 
\grantsponsor{FWF}{Austrian Science Fund}{www.fwf.ac.at} under Grant
\grantnum[www.fwf.ac.at]{FWF}{P 28287}
and by
``Vienna Scientific Cluster'' (VSC) for providing computational resource

I would like to thank the referees for several helpful suggestions which greatly improved the presentation of this paper.
\end{acks}


\newpage

\bibliographystyle{ACM-Reference-Format}
\setcitestyle{numbers,sort&compress}
\bibliography{biblio}


\begin{thebibliography}{31}


\ifx \showCODEN    \undefined \def \showCODEN     #1{\unskip}     \fi
\ifx \showDOI      \undefined \def \showDOI       #1{#1}\fi
\ifx \showISBNx    \undefined \def \showISBNx     #1{\unskip}     \fi
\ifx \showISBNxiii \undefined \def \showISBNxiii  #1{\unskip}     \fi
\ifx \showISSN     \undefined \def \showISSN      #1{\unskip}     \fi
\ifx \showLCCN     \undefined \def \showLCCN      #1{\unskip}     \fi
\ifx \shownote     \undefined \def \shownote      #1{#1}          \fi
\ifx \showarticletitle \undefined \def \showarticletitle #1{#1}   \fi
\ifx \showURL      \undefined \def \showURL       {\relax}        \fi
\providecommand\bibfield[2]{#2}
\providecommand\bibinfo[2]{#2}
\providecommand\natexlab[1]{#1}
\providecommand\showeprint[2][]{arXiv:#2}

\bibitem[\protect\citeauthoryear{Ahmadi, Jungers, Parrilo, and
  Roozbehani}{Ahmadi et~al\mbox{.}}{2011}]%
        {AJPR11}
\bibfield{author}{\bibinfo{person}{A. Ahmadi}, \bibinfo{person}{Rapha{\"e}l
  Jungers}, \bibinfo{person}{Pablo~A. Parrilo}, {and}
  \bibinfo{person}{Mardavjij Roozbehani}.} \bibinfo{year}{2011}\natexlab{}.
\newblock \showarticletitle{Joint Spectral Radius and Path-Complete Graph
  Lyapunov Functions}.
\newblock \bibinfo{journal}{\emph{SIAM J. Control Optim.}}
  \bibinfo{volume}{52}, \bibinfo{number}{1} (\bibinfo{year}{2011}),
  \bibinfo{pages}{687--717}.
\newblock


\bibitem[\protect\citeauthoryear{Barabanov}{Barabanov}{1988}]%
        {Bar88}
\bibfield{author}{\bibinfo{person}{N.~E. Barabanov}.}
  \bibinfo{year}{1988}\natexlab{}.
\newblock \showarticletitle{Lyapunov indicator for discrete inclusions I--III}.
\newblock \bibinfo{journal}{\emph{Autom. Remote Control}} \bibinfo{volume}{49},
  \bibinfo{number}{2} (\bibinfo{year}{1988}), \bibinfo{pages}{152--157}.
\newblock


\bibitem[\protect\citeauthoryear{Berger and Wang}{Berger and Wang}{1992}]%
        {BW92}
\bibfield{author}{\bibinfo{person}{Marc~A. Berger} {and} \bibinfo{person}{Yang
  Wang}.} \bibinfo{year}{1992}\natexlab{}.
\newblock \showarticletitle{Bounded semigroups of matrices}.
\newblock \bibinfo{journal}{\emph{Linear Alg. Appl.}}  \bibinfo{volume}{166}
  (\bibinfo{year}{1992}), \bibinfo{pages}{21--27}.
\newblock


\bibitem[\protect\citeauthoryear{Blondel and Chang}{Blondel and Chang}{2011}]%
        {BC11}
\bibfield{author}{\bibinfo{person}{Vincent~D. Blondel} {and}
  \bibinfo{person}{Chia-Tche Chang}.} \bibinfo{year}{2011}\natexlab{}.
\newblock \bibinfo{title}{A genetic algorithm approach for the approximation of
  the joint spectral radius}.
\newblock
  \bibinfo{howpublished}{\url{https://perso.uclouvain.be/chia-tche.chang/code.php}}.
\newblock


\bibitem[\protect\citeauthoryear{Blondel and Chang}{Blondel and Chang}{2013}]%
        {BC13}
\bibfield{author}{\bibinfo{person}{Vincent~D. Blondel} {and}
  \bibinfo{person}{Chia-Tche Chang}.} \bibinfo{year}{2013}\natexlab{}.
\newblock \showarticletitle{An experimental study of approximation algorithms
  for the joint spectral radius}.
\newblock \bibinfo{journal}{\emph{Numer. Algor.}}  \bibinfo{volume}{64}
  (\bibinfo{year}{2013}), \bibinfo{pages}{181--202}.
\newblock


\bibitem[\protect\citeauthoryear{Blondel and Jungers}{Blondel and
  Jungers}{2008}]%
        {BJ08}
\bibfield{author}{\bibinfo{person}{Vincent~D. Blondel} {and}
  \bibinfo{person}{Rapha{\"e}l Jungers}.} \bibinfo{year}{2008}\natexlab{}.
\newblock \showarticletitle{On the finiteness property for rational matrices}.
\newblock \bibinfo{journal}{\emph{Linear Alg. Appl.}} \bibinfo{volume}{428},
  \bibinfo{number}{10} (\bibinfo{year}{2008}), \bibinfo{pages}{2283--2295}.
\newblock


\bibitem[\protect\citeauthoryear{Blondel, Jungers, and Protasov}{Blondel
  et~al\mbox{.}}{2006}]%
        {BJP06}
\bibfield{author}{\bibinfo{person}{Vincent~D. Blondel},
  \bibinfo{person}{Rapha{\"e}l Jungers}, {and} \bibinfo{person}{Vladimir~Yu.
  Protasov}.} \bibinfo{year}{2006}\natexlab{}.
\newblock \showarticletitle{On the Complexity of Computing the Capacity of
  Codes That Avoid Forbidden Difference Patterns}.
\newblock \bibinfo{journal}{\emph{IEEE Trans. Inf. Theory}}
  \bibinfo{volume}{52} (\bibinfo{year}{2006}), \bibinfo{pages}{5122--5127}.
\newblock


\bibitem[\protect\citeauthoryear{Blondel, Jungers, and Protasov}{Blondel
  et~al\mbox{.}}{2010}]%
        {BJP10}
\bibfield{author}{\bibinfo{person}{Vincent~D. Blondel},
  \bibinfo{person}{Rapha{\"e}l Jungers}, {and} \bibinfo{person}{Vladimir~Yu.
  Protasov}.} \bibinfo{year}{2010}\natexlab{}.
\newblock \showarticletitle{Joint spectral characteristics of matrices: a conic
  programming approach}.
\newblock \bibinfo{journal}{\emph{SIAM J. Matr. Anal. Appl.}}
  \bibinfo{volume}{31}, \bibinfo{number}{4} (\bibinfo{year}{2010}),
  \bibinfo{pages}{2146--2162}.
\newblock


\bibitem[\protect\citeauthoryear{Blondel, Nesterov, and Theys}{Blondel
  et~al\mbox{.}}{2005}]%
        {BN05}
\bibfield{author}{\bibinfo{person}{Vincent~D. Blondel}, \bibinfo{person}{Yurii
  Nesterov}, {and} \bibinfo{person}{Jacques Theys}.}
  \bibinfo{year}{2005}\natexlab{}.
\newblock \showarticletitle{On the accuracy of the ellipsoid norm approximation
  of the joint spectral radius}.
\newblock \bibinfo{journal}{\emph{Linear Algebra Appl.}} \bibinfo{volume}{394},
  \bibinfo{number}{1} (\bibinfo{year}{2005}), \bibinfo{pages}{91--107}.
\newblock


\bibitem[\protect\citeauthoryear{Blondel and Tsitsiklis}{Blondel and
  Tsitsiklis}{1997}]%
        {BT97}
\bibfield{author}{\bibinfo{person}{Vincent~D. Blondel} {and}
  \bibinfo{person}{John~N. Tsitsiklis}.} \bibinfo{year}{1997}\natexlab{}.
\newblock \showarticletitle{The Lyapunov exponent and joint spectral radius of
  pairs of matrices are hard -- when not impossible -- to compute and to
  approximate}.
\newblock \bibinfo{journal}{\emph{Math. Control Sign. Syst.}}
  \bibinfo{volume}{10}, \bibinfo{number}{1} (\bibinfo{year}{1997}),
  \bibinfo{pages}{31--40}.
\newblock


\bibitem[\protect\citeauthoryear{Blondel and Tsitsiklis}{Blondel and
  Tsitsiklis}{2000}]%
        {BT00}
\bibfield{author}{\bibinfo{person}{Vincent~D. Blondel} {and}
  \bibinfo{person}{John~N. Tsitsiklis}.} \bibinfo{year}{2000}\natexlab{}.
\newblock \showarticletitle{The boundedness of all products of a pair of
  matrices is undecidable}.
\newblock \bibinfo{journal}{\emph{Syst. Control Lett.}} \bibinfo{volume}{41},
  \bibinfo{number}{2} (\bibinfo{year}{2000}), \bibinfo{pages}{135--140}.
\newblock


\bibitem[\protect\citeauthoryear{Charina and Mejstrik}{Charina and
  Mejstrik}{2018}]%
        {CM18}
\bibfield{author}{\bibinfo{person}{Maria Charina} {and} \bibinfo{person}{Thomas
  Mejstrik}.} \bibinfo{year}{2018}\natexlab{}.
\newblock \showarticletitle{Multiple multivariate subdivision schemes: matrix
  and operator approaches}.
\newblock \bibinfo{journal}{\emph{Comput. Appl. Math.}}  \bibinfo{volume}{349}
  (\bibinfo{year}{2018}), \bibinfo{pages}{279--291}.
\newblock


\bibitem[\protect\citeauthoryear{Charina and Protasov}{Charina and
  Protasov}{2019}]%
        {CP17}
\bibfield{author}{\bibinfo{person}{Maria Charina} {and}
  \bibinfo{person}{Vladimir~Yu. Protasov}.} \bibinfo{year}{2019}\natexlab{}.
\newblock \showarticletitle{Regularity of anisotropic refinable functions}.
\newblock \bibinfo{journal}{\emph{Appl. Comput. Harm. A.}}
  \bibinfo{volume}{47}, \bibinfo{number}{3} (\bibinfo{year}{2019}),
  \bibinfo{pages}{795--821}.
\newblock


\bibitem[\protect\citeauthoryear{Daubechies}{Daubechies}{1988}]%
        {D88}
\bibfield{author}{\bibinfo{person}{Ingrid Daubechies}.}
  \bibinfo{year}{1988}\natexlab{}.
\newblock \showarticletitle{Orthonormal bases of compactly supported wavelets}.
\newblock \bibinfo{journal}{\emph{Comm. Pure Appl. Math.}}
  \bibinfo{volume}{41} (\bibinfo{year}{1988}).
\newblock


\bibitem[\protect\citeauthoryear{Daubechies and Lagarias}{Daubechies and
  Lagarias}{1992}]%
        {DL1992}
\bibfield{author}{\bibinfo{person}{Ingrid Daubechies} {and}
  \bibinfo{person}{Jeffrey~C. Lagarias}.} \bibinfo{year}{1992}\natexlab{}.
\newblock \showarticletitle{Two-scale difference equations. ii. local
  regularity, infinite products of matrices and fractals}.
\newblock \bibinfo{journal}{\emph{SIAM J. Math. Anal.}} \bibinfo{volume}{23},
  \bibinfo{number}{4} (\bibinfo{year}{1992}), \bibinfo{pages}{1031--1079}.
\newblock


\bibitem[\protect\citeauthoryear{Gripenberg}{Gripenberg}{1996}]%
        {Grip96}
\bibfield{author}{\bibinfo{person}{Gustav Gripenberg}.}
  \bibinfo{year}{1996}\natexlab{}.
\newblock \showarticletitle{Computing the joint spectral radius}.
\newblock \bibinfo{journal}{\emph{Linear Alg. Appl.}}  \bibinfo{volume}{234}
  (\bibinfo{year}{1996}), \bibinfo{pages}{43--60}.
\newblock


\bibitem[\protect\citeauthoryear{Guglielmi and Protasov}{Guglielmi and
  Protasov}{2013}]%
        {GP13}
\bibfield{author}{\bibinfo{person}{Nicola Guglielmi} {and}
  \bibinfo{person}{Vladimir~Yu. Protasov}.} \bibinfo{year}{2013}\natexlab{}.
\newblock \showarticletitle{Exact Computation of Joint Spectral Characteristics
  of Linear Operators}.
\newblock \bibinfo{journal}{\emph{Found. Comput. Math.}}  \bibinfo{volume}{13}
  (\bibinfo{year}{2013}), \bibinfo{pages}{37--39}.
\newblock


\bibitem[\protect\citeauthoryear{Guglielmi and Protasov}{Guglielmi and
  Protasov}{2015}]%
        {GP15poin}
\bibfield{author}{\bibinfo{person}{Nicola Guglielmi} {and}
  \bibinfo{person}{Vladimir~Yu. Protasov}.} \bibinfo{year}{2015}\natexlab{}.
\newblock \showarticletitle{Matrix approach to the global and local regularity
  of wavelets}.
\newblock \bibinfo{journal}{\emph{Poincare J. Anal. Appl.}}
  \bibinfo{volume}{2} (\bibinfo{year}{2015}), \bibinfo{pages}{77--92}.
\newblock


\bibitem[\protect\citeauthoryear{Guglielmi and Protasov}{Guglielmi and
  Protasov}{2016}]%
        {GP16}
\bibfield{author}{\bibinfo{person}{Nicola Guglielmi} {and}
  \bibinfo{person}{Vladimir~Yu. Protasov}.} \bibinfo{year}{2016}\natexlab{}.
\newblock \showarticletitle{Invariant polytopes of linear operators with
  applications to regularity of wavelets and of subdivisions}.
\newblock \bibinfo{journal}{\emph{SIAM J. Matrix Anal. \& Appl.}}
  \bibinfo{volume}{37}, \bibinfo{number}{1} (\bibinfo{year}{2016}),
  \bibinfo{pages}{18--52}.
\newblock


\bibitem[\protect\citeauthoryear{Guglielmi, Wirth, and Zennaro}{Guglielmi
  et~al\mbox{.}}{2005}]%
        {GWZ05}
\bibfield{author}{\bibinfo{person}{Nicola Guglielmi}, \bibinfo{person}{Fabian
  Wirth}, {and} \bibinfo{person}{Marco Zennaro}.}
  \bibinfo{year}{2005}\natexlab{}.
\newblock \showarticletitle{Complex polytope extremality results for families
  of matrices}.
\newblock \bibinfo{journal}{\emph{SIAM J. Matrix Anal. Appl.}}
  \bibinfo{volume}{27}, \bibinfo{number}{3} (\bibinfo{year}{2005}),
  \bibinfo{pages}{721--743}.
\newblock


\bibitem[\protect\citeauthoryear{Guglielmi and Zennaro}{Guglielmi and
  Zennaro}{2008}]%
        {GZ08}
\bibfield{author}{\bibinfo{person}{Nicola Guglielmi} {and}
  \bibinfo{person}{Marco Zennaro}.} \bibinfo{year}{2008}\natexlab{}.
\newblock \showarticletitle{An algorithm for finding extremal polytope norms of
  matrix families}.
\newblock \bibinfo{journal}{\emph{Linear Alg. Appl.}} \bibinfo{volume}{428},
  \bibinfo{number}{10} (\bibinfo{year}{2008}), \bibinfo{pages}{2265--2282}.
\newblock


\bibitem[\protect\citeauthoryear{Guglielmi and Zennaro}{Guglielmi and
  Zennaro}{2009}]%
        {GZ09}
\bibfield{author}{\bibinfo{person}{Nicola Guglielmi} {and}
  \bibinfo{person}{Marco Zennaro}.} \bibinfo{year}{2009}\natexlab{}.
\newblock \showarticletitle{Finding extremal complex polytope norms for
  families of real matrices}.
\newblock \bibinfo{journal}{\emph{SIAM J. Matrix Anal. Appl.}}
  \bibinfo{volume}{31}, \bibinfo{number}{2} (\bibinfo{year}{2009}),
  \bibinfo{pages}{602--620}.
\newblock


\bibitem[\protect\citeauthoryear{Gurvits}{Gurvits}{1995}]%
        {Gur95}
\bibfield{author}{\bibinfo{person}{Leonid Gurvits}.}
  \bibinfo{year}{1995}\natexlab{}.
\newblock \showarticletitle{Stability of discrete linear inclusion}.
\newblock \bibinfo{journal}{\emph{Linear Alg. Appl.}}  \bibinfo{volume}{231}
  (\bibinfo{year}{1995}), \bibinfo{pages}{47--85}.
\newblock


\bibitem[\protect\citeauthoryear{Hare, Morris, Sidorov, and Theys}{Hare
  et~al\mbox{.}}{2011}]%
        {HMST11}
\bibfield{author}{\bibinfo{person}{Kevin~G. Hare}, \bibinfo{person}{Ian~D.
  Morris}, \bibinfo{person}{Nikita Sidorov}, {and} \bibinfo{person}{Jacques
  Theys}.} \bibinfo{year}{2011}\natexlab{}.
\newblock \showarticletitle{An explicit counterexample to the Lagarias--Wang
  finiteness conjecture}.
\newblock \bibinfo{journal}{\emph{Adv. Math.}} \bibinfo{volume}{226},
  \bibinfo{number}{6} (\bibinfo{year}{2011}), \bibinfo{pages}{4667--4701}.
\newblock


\bibitem[\protect\citeauthoryear{Hendrickx, Jungers, and
  Vankeerberghen}{Hendrickx et~al\mbox{.}}{2014}]%
        {Jung2014}
\bibfield{author}{\bibinfo{person}{Julien~M. Hendrickx},
  \bibinfo{person}{Rapha{\"e}l Jungers}, {and} \bibinfo{person}{Guillaume
  Vankeerberghen}.} \bibinfo{year}{2014}\natexlab{}.
\newblock \bibinfo{title}{JSR: A Toolbox to Compute the Joint Spectral Radius}.
\newblock
  \bibinfo{howpublished}{\url{mathworks.com/matlabcentral/fileexchange/33202}}.
\newblock


\bibitem[\protect\citeauthoryear{Kozyakin}{Kozyakin}{2010}]%
        {Koz10b}
\bibfield{author}{\bibinfo{person}{Victor~S. Kozyakin}.}
  \bibinfo{year}{2010}\natexlab{}.
\newblock \showarticletitle{Iterative building of Barabanov norms and
  computation of the joint spectral radius for matrix sets}.
\newblock \bibinfo{journal}{\emph{Discrete Continuous Dyn. Syst. Ser. B}}
  \bibinfo{volume}{14}, \bibinfo{number}{1} (\bibinfo{year}{2010}),
  \bibinfo{pages}{143--158}.
\newblock


\bibitem[\protect\citeauthoryear{Möller and Reif}{Möller and Reif}{2014}]%
        {MR01}
\bibfield{author}{\bibinfo{person}{Claudia Möller} {and}
  \bibinfo{person}{Ulrich Reif}.} \bibinfo{year}{2014}\natexlab{}.
\newblock \showarticletitle{A tree-based approach to joint spectral radius
  determination}.
\newblock \bibinfo{journal}{\emph{Linear Alg. Appl.}}  \bibinfo{volume}{463}
  (\bibinfo{year}{2014}), \bibinfo{pages}{154--170}.
\newblock


\bibitem[\protect\citeauthoryear{Moision, Orlitsky, and Siegel}{Moision
  et~al\mbox{.}}{2001}]%
        {MOS01}
\bibfield{author}{\bibinfo{person}{Bruce~E. Moision}, \bibinfo{person}{Alon
  Orlitsky}, {and} \bibinfo{person}{Paul~H. Siegel}.}
  \bibinfo{year}{2001}\natexlab{}.
\newblock \showarticletitle{On codes that avoid specified differences}.
\newblock \bibinfo{journal}{\emph{IEEE Trans. Inf. Theory}}
  \bibinfo{volume}{47} (\bibinfo{year}{2001}), \bibinfo{pages}{433--442}.
\newblock


\bibitem[\protect\citeauthoryear{Parrilo and Jadbabaie}{Parrilo and
  Jadbabaie}{2008}]%
        {PJ08}
\bibfield{author}{\bibinfo{person}{Pablo~A. Parrilo} {and} \bibinfo{person}{Ali
  Jadbabaie}.} \bibinfo{year}{2008}\natexlab{}.
\newblock \showarticletitle{Approximation of the joint spectral radius using
  sum of squares}.
\newblock \bibinfo{journal}{\emph{Linear Alg. Appl.}} \bibinfo{volume}{428},
  \bibinfo{number}{10} (\bibinfo{year}{2008}), \bibinfo{pages}{2385--2402}.
\newblock


\bibitem[\protect\citeauthoryear{Protasov}{Protasov}{2000}]%
        {Prot00}
\bibfield{author}{\bibinfo{person}{Vladimir~Yu. Protasov}.}
  \bibinfo{year}{2000}\natexlab{}.
\newblock \showarticletitle{Asymptotic behaviour of the partition function}.
\newblock \bibinfo{journal}{\emph{Sb. Math.}} \bibinfo{volume}{191},
  \bibinfo{number}{3--4} (\bibinfo{year}{2000}), \bibinfo{pages}{230--233}.
\newblock


\bibitem[\protect\citeauthoryear{Rota and Strang}{Rota and Strang}{1960}]%
        {Rota60}
\bibfield{author}{\bibinfo{person}{Gian-Carlo Rota} {and}
  \bibinfo{person}{Gilbert~W. Strang}.} \bibinfo{year}{1960}\natexlab{}.
\newblock \showarticletitle{A note on the joint spectral radius}.
\newblock \bibinfo{journal}{\emph{Kon. Nederl. Acad. Wet. Proc.}}
  \bibinfo{volume}{63} (\bibinfo{year}{1960}), \bibinfo{pages}{379--381}.
\newblock


\end{thebibliography}

\end{document}